\documentclass[12pt]{amsart}
\usepackage{amsmath,amssymb, amscd, graphicx}

\usepackage{graphicx}
\usepackage{amscd}
\usepackage{graphics}

\makeatletter
\@addtoreset{equation}{section}
\makeatother
\usepackage{enumerate}

\def\hX{{\hat{X}}}

\def\cV{{\mathcal{V}}}

\def\cL{{\mathcal{L}}}
\def\cP{{\mathcal{P}}}
\def\tDelta{{\Delta^{\mathrm{per}}}}

\def\hy{{\hat{y}{}}}
\def\ty{{\tilde{y}{}}}

\def\ty{{\tilde{y}}}

\def\hw{{\hat{w}}}

\def\ee{\mathrm e}
\def\Amp{w}

\def\Sym{\mathrm {S}}

\def\cI{{\mathcal{I}}}

\def\CC{{\mathbb C}}
\def\FF{{\mathbb F}}
\def\JJ{{\mathcal J}}
\def\ZZ{{\mathbb Z}}

\def\cO{{\mathcal{O}}}

\def\WW{{\mathcal{W}}}

\def\nuI{{{\nu^{\rm I}}}}

\def\nuII{{{\nu^{\rm II}}}}

\def\ugo{{u_1}}

\def\LF{\lfloor}
\def\RF{\rfloor}

\newtheorem{definition}{Definition}[section]

\newtheorem{theorem}[definition]{Theorem}
\newtheorem{proposition}[definition]{Proposition}
\newtheorem{corollary}[definition]{Corollary}
\newtheorem{remark}[definition]{Remark}
\newtheorem{lemma}[definition]{Lemma}

\newtheorem{example}[definition]{Example}

\def\dfrac#1#2{{\displaystyle\frac{#1}{#2}}}
\def\dint{{\displaystyle\int}}

\def\book#1{\rm{#1}, }
\def\paper#1{\textit{#1}, }
\def\jour#1{\rm{#1}, }
\def\yr#1{({\rm{#1}) }}

\def\vol#1{\textbf{#1}}
\def\pages#1{\rm{#1}}
\def\page#1{\rm{#1}}

\def\publaddr#1{\rm{#1}, }
\def\publ#1{\rm{#1}, }
\def\by#1{{\rm{#1}, }}

\def\eds{\rm{eds.}}

\def\ft{u}

\begin{document}

\title{Quasi-periodic and periodic solutions of the 
Toda lattice via the hyperelliptic sigma function}

\author{Yuji Kodama, Shigeki Matsutani, Emma Previato}


\maketitle

\begin{abstract}
M. Toda in 1967 (\textit{J. Phys. Soc. Japan}, \textbf{22} and
\textbf{23}) considered a lattice model with exponential
interaction and proved, as suggested by the Fermi-Pasta-Ulam experiments
in the 1950s,
that it has exact periodic and soliton solutions.
The Toda lattice, as it came to be known, was then extensively studied
as one of the completely integrable (differential-difference)
non-linear equations which admit exact solutions in terms
of theta functions of hyperelliptic curves.
In this paper, we extend 
Toda's  original approach to give hyperelliptic solutions
of the Toda lattice 
in terms of hyperelliptic
Kleinian (sigma) functions for arbitrary genus. 
The key identities are 
given by generalized addition formulae for the hyperelliptic
sigma functions (J.C. Eilbeck \textit{et al.},
{\it J. reine angew. Math.}
{\bf 619}, 2008).
We then show that 
periodic (in the discrete variable, a standard term in the Toda lattice
theory) 
solutions of the Toda lattice correspond to the zeros of
Kiepert-Brioschi's division polynomials, and note
these are related to solutions of Poncelet's closure problem.
One feature of our solution is that the hyperelliptic curve 
is related in a non-trivial way to the one previously used.
\end{abstract}

\bigskip

\tableofcontents

\section{Introduction}

The Toda lattice is an ``algebraically completely integrable system''.
As such, it admits classes of solutions parametrized by Jacobi
varieties of compact Riemann surfaces (or algebraic curves),
the ``algebro-geometric solitons'' \cite{GHMT}.  One advantage of
the algebro-geometric solution is that the time flows become
linear on a hyperelliptic Jacobian, and the difference operator
is translation on the Jacobian.
These are the objects of concern here.

On the other hand, 
recent work \cite{EEMOP, MP} on Kleinian sigma functions has
focused on addition formulae on a certain
stratification of the Jacobian; given the relevance of this stratification
in terms of the orbits  of the Toda lattice \cite{AM,AHM,CK,Ko},
our program is to study the explicit relationship between the two,
with the result of an exchange of knowledge.
 
In this first paper we identify and study  the
(quasi-)periodic  solutions of the 
Toda lattice equation in terms of the hyperelliptic sigma functions.
Our approach is somewhat different from  the one that
exists in the literature,  and in particular
it gives us a different `spectral curve' and algebraic
conditions for periodicity.
To give a sketch of our method  and results, 
we first say a few words about
the history and the significance of the $\sigma$-function,
especially in the field of differential/difference equations.
In the 19th century,  it was the study of Riemann-surface
theory that led to the associated algebraic (meromorphic) or
analytic (theta) functions and the (differential)
equations that they satisfy, apart some exceptional cases
In keeping with such philosophy,
 Baker in \cite{B3} discovered the KdV hierarchy and KP equation\footnote{KP
appears in the case when the affine model of the
hyperelliptic curve has two points at infinity \cite{Mat2}.} for
every hyperelliptic curve of genus $g$ as an application of 
Kleinian sigma functions  (1903) and posed this challenge:
\begin{quote}
{{These equations put a {\it problem}: To obtain
a theory of differential equations which shall shew
from them why, if we assume
\begin{gather*}
	\wp_{\lambda\mu}(u) =
        -\partial^2\log \sigma(u)/
         \partial u_\lambda \partial u_\mu,
\end{gather*}
the function $\sigma(u)$ has the properties
which {\it a priori} we know it to possess,
and how far the forms of the equations
are essential to these properties.}}
\end{quote}
Baker extensively studied the 
differential equations satisfied by the $\sigma$ function.
In the late 1960s and early 1970s, on the contrary,
many authors started from the (``completely integrable'')
differential equations, and arrived at a spectral curve
which is a Riemann surface and carries a ``Baker-Akhiezer'' 
eigenfunction for the (linearizing) ``Lax-pair'' equations.
Mumford in his book on theta functions
\cite{Mu} demonstrated the close relationship between
 the differential equations and the
algebraic approach, together with the abelian
function theory of the 19th century based on Jacobi's theory;
he gave an explicit dictionary 
between the theta function for hyperelliptic
curves and three polynomials in one variable,	sometimes 
called the Mumford triplet (whose definition he attributes to Jacobi),
which parametrize the Jacobi variety, cf. Remark \ref{rmk:BakerFayToda}.

Recently the Kleinian sigma function was reprised, generalized
and studied by several authors \cite{BEL, EEL}.
These authors showed that sigma is more efficient than Riemann's
theta function (in fact, the sigma function 
approaches the Schur polynomials in the limit when the curve becomes rational)
to solve differential equations.
This is our point of view.

We add, as suggested by the referee, a comment on the geometric significance 
of sigma: for more detail we refer to the aforementioned studies.
Since sigma vanishes on the (higher) Abel images of the curve $X_g$,
it is better suited to be expanded in the abelian variables which correspond 
to the hyperosculating flows (Section \ref{Sigma}) to the image of $X_g$,
which are the flows of the KP hierarchy.  It is in these variables that sigma
equals a Schur polynomial\footnote{Schur polynomials
were used by Sato to define his tau function and originally construct
the universal solution to the KP hierarchy.
The relationship of sigma with tau is pursued in
 \cite{EEB}.} up to higher-order terms \cite{N}. In fact
expansions and computer-algebra work have enabled guesses and proofs
of the addition formulas which generalize (\ref{eq:add1}) and
are essential to solving integrable
equations. Lastly, sigma (unlike Riemann's theta, cf. \cite[II.5]{Mu})
is invariant under the action of the modular group on the period matrix
of $X_g$.

We also comment on the fact that our addition formulae, which are key to
the solution as explained in the next paragraph, are particular to
hyperelliptic curves.  Of course addition theorems hold for general
curves, cf. e.g. Theorems 9.1 and 10.1
 in \cite{EEMOP2} for the trigonal case; however,
they will not be as expedient as the hyperelliptic ones,
for example an expression in terms of algebraic functions on the curve
is as yet unwieldy, as perhaps can be expected, since the non-hyperelliptic
Neumann systems  have polynomial Hamiltonians of a very complicated type
\cite{Sch1, Sch2, Sch3}.  The basic reason why the hyperelliptic
case simplifies, and in fact was the one for which Baker 
was able to study explicitly the sigma function, is the 
hyperelliptic involution which lifts to the Jacobian in such a way that 
it acts on  sigma only by a sign.

In this article, we consider the solutions of the, one- and 
two-(time-)dimensional,
Toda one-(space-)dimensional lattice.
 Toda gave an exact solution for the 
 Toda lattice equation using an identity of 
 elliptic-function theory \cite{To}.
For an arbitrary hyperelliptic curve of genus $g$,
we construct a meromorphic function on the Jacobian of the curve
that obeys the Toda lattice
equation over the Jacobian of the curve, by
using an identity of hyperelliptic abelian functions
(see Theorems \ref{thm:Toda} and \ref{thm:twoToda}  
 for the one- and two-(time-)dimensional
Toda lattice equations respectively.).
Remark  \ref{rmk:BakerFayToda}
shows that the two-(time-)dimensional Toda equation is
equivalent to a relation which generates 
Fay's trisecant formula and Baker's derivation of the KdV hierarchy
and the KP equation for every genus.
Despite this equivalence, we stress again here the geometric nature
of sigma compared to theta. That nature gives a dictionary between
abelian and 
 algebraic
functions of the curve, specifically, the affine coordinates of
its planar representation, as well as 
algebraic coefficients for the relations in the ring of differential
operators
that act on the Jacobian
(more details are given in Remark \ref{RevSigma}
and \ref{rmk:BakerFayToda}(2)(b)).  Weierstrass  recognised this fact in 
his study of hyperelliptic theta functions \cite{W1}; in fact, he defined 
the al and Al functions (named after Abel) which can be used
to give solutions of  finite-dimensional Hamiltonian systems \cite{Ma4}.

We obtain  quasi-periodic\footnote{As customary in the literature on
the Toda lattice, cf. the
monographs \cite{AvMV} and \cite{V},
the word ``periodic'' does not refer to the time,
but rather the space variable of the lattice sites.} (hyperelliptic)
solutions of the Toda lattice, since \textit{a priori} the
discrete variable 
has no rational relation with the period lattice of the curve.
Finding periodic solutions is equivalent 
to giving a torsion point of the curve, which satisfies  
a ``division polynomial'',
 known as  Kiepert's or Brioschi's polynomial.
Using a zero of the division polynomial (if it exists),
we construct a hyperelliptic solution of Toda of genus $g$ with
 period $N(>g)$ in 
Theorems \ref{thm:pToda} and \ref{thm:KvM_psi}.
We illustrate the $g=1$ case as an example and in that case
point out  that the relation between the division polynomial and
Toda lattice is the same as 
Poncelet's closure (Appendix).

\medskip
{Acknowledgements.}
One of authors (S.M.) thanks Boris Mirman for bringing 
Poncelet's problem to his attention
and Akira Ohbuchi who drew his attention to
Galois's study on fifth elliptic cyclic point.
Y.K. is partially supported by NSF grant DMS-0806219.
E.P. acknowledges very valuable partial research support under grant
NSF-DMS-0808708. We are indebted to the referee for corrections and
suggestions.

\section{Genus-one case}\label{Genusone}
In this Section, we demonstrate how one gets an elliptic solution of the 
Toda lattice \cite{To, KvM, Ma3}.
Let $X_1$ be  an elliptic curve given by
\begin{gather}
X_1:\
  y^2 =x^3 + \lambda_2 x^2
                +\lambda_1 x + \lambda_0\nonumber \\
\phantom{C_1:\ \frac{1}{4} \hy{}^2=y^2}= (x-e_1)(x-e_2)(x-e_3),
\label{eq:ecurve}
\end{gather}
where the $e$'s are distinct
complex numbers and $\lambda_2 =-(e_1+e_2+e_3)= 0$.

The Weierstrass elliptic $\sigma$ function
associated with the curve $X_1$  is
connected with the Weierstrass $\wp$ and $\zeta$ functions
by
\begin{equation}
     \wp(u) =- \frac{d^2}{d u^2} \log \sigma(u), \quad
     \zeta(u) = \frac{d}{d u} \log \sigma(u).
\label{eq:wpzeta}
\end{equation}
The Jacobian of the curve $X_1$ 
is given by $\mathcal{J}_1 = \CC/(\ZZ \omega' + \ZZ \omega'')$
using the periods $(\omega', \omega'')$, complete integrals
of the first kind.
The abelian
coordinate $u$, actually defined on the universal cover of the Jacobian
$\mathcal{J}_1$ but
used as customary up to congruence when no ambiguity arises,
is given by
\begin{equation}
     u = \int^{(x,y)}_\infty \frac{d x}{2 y},
\end{equation}
with $x(u)=\wp(u),~ 2 y(u)=\wp'(u)$ and $\infty$
the infinity point of $X_1$.

The key to obtain a $\wp$-function solution of the Toda lattice
is the addition formula,
\begin{equation}
        \wp(u) - \wp(v) = -
\frac{\sigma(u+v)\sigma(u-v)}{[\sigma(v)\sigma(u)]^2}.
\label{eq:add1}
\end{equation}
By differentiating the logarithm of (\ref{eq:add1}) with respect to $u$
twice, we have
\begin{equation}
- \frac{d^2}{d u^2} \log [\wp(u)-\wp(v)]
         = \wp(u+v) - 2\wp(u) + \wp(u-v).
\end{equation}
Thus for a constant number $u_0$ and any integer $n$,
by letting $u=n u_0+t + t_0$, $v=u_0$, 
\begin{gather}
\begin{split}
        -\frac{d^2}{d t^2} \log [\wp(n u_0+t + t_0)-\wp(u_0)]
        &= [\wp((n+1)u_0+t + t_0)-\wp(u_0)] \\
&- 2[\wp(n u_0+t+t_0)-\wp(u_0)]
             + [\wp((n-1)u_0+t+t_0)-\wp(u_0)].
\end{split}
\end{gather}
Which, by letting
$V_n(t) := -\wp(nu_0+t + t_0)$,
$V_c:= -\wp(u_0)$ and
$q_n := -\log [V_{n}(t)-V_c]$,
becomes
\begin{equation}
-\frac{d^2}{d t^2}q_n = \ee^{-q_{n+1}}-2 \ee^{-q_n}
+\ee^{-q_{n-1}} \qquad (n \in \ZZ).
\end{equation}
This can be identified with the continuous Toda lattice equation,
where $q_n$ are the interaction potentials.
Indeed, by letting $q_n = Q_{n} - Q_{n-1}$, where $Q_n$ is the displacement
of the $n$-th particle
and obeys the nonlinear differential equation of the exponential
lattice~\cite{To},
\begin{equation}
-\frac{d^2}{d t^2}Q_n = \ee^{Q_n-Q_{n+1}}- \ee^{Q_{n-1}-Q_{n}}\qquad 
(n \in \ZZ).
\end{equation}
To  provide the connection 
between Toda's original solution for $e_1,e_2,e_3\in \mathbb{R}$ \cite{To}
and  this elliptic solution derived from the
addition formula (\ref{eq:add1}) with (\ref{eq:wpzeta}), 
we observe that,
by letting $t_0 = -\omega''$,
$$
\wp(t + n u_0 - \omega'') -\wp(t + n u_0)=
(e_1 - e_3)\  \mathrm{ns}^2((e_1 - e_3)^{1/2}(t + n u_0))  + e_3,
$$
where $\mathrm{ns}(u)=k\mathrm{sn}(u+\imath K^\prime )$
and the modulus of the Jacobi elliptic functions
is $k^2=(e_2-e_3)/(e_1-e_3)$ \cite[22$\cdot$351]{WW}.

\section{Hyperelliptic curve $X_g$ and sigma functions}\label{Sigma}

In this Section, we give  background information on 
the hyperelliptic $\theta$ functions and
the $\sigma$ function, a generalization to higher genus of 
the Weierstrass elliptic $\sigma$ function.

\subsection{Geometrical setting for hyperelliptic curves}
Let $X_g$ be a hyperelliptic curve defined by
$$
	X_g~:~ y^2 = f(x):= x^{2g+1} + \lambda_{2g} x^{2g} + \cdots +\lambda_0
$$
where $\lambda_j$'s are complex numbers,
together with a smooth point  $\infty$  at infinity.  
Let the affine ring related to $X_g$ be 
$R_g:=\CC[x, y]/(y^2 - f(x))$.
We fix a basis of holomorphic one-forms 
\[ 
    \nuI_i =\frac{x^{i-1} d x}{2 y} \qquad (i=1,\ldots,g).
\]
We also fix a homology basis for the curve  $X_g$  so that
$$
\mathrm{H}_1(X_g, \mathbb Z)
  =\bigoplus_{j=1}^g\mathbb Z\alpha_{j}
   \oplus\bigoplus_{j=1}^g\mathbb Z\beta_{j},
$$
with the intersections  given by
$[\alpha_i, \alpha_j]=0$, $[\beta_i, \beta_j]=0$ and
$[\alpha_i, \beta_j]=-[\beta_i, \alpha_j]=\delta_{ij}$.
We consider the half-period matrix 
$\omega=\left[\begin{matrix}\omega'\\ \omega''\end{matrix}\right]$
 of  $X_g$  with respect to the given basis where
$$
    \omega'=\frac{1}{2}\left[\oint_{\alpha_{j}}\nuI_{i}\right],
\quad
      \omega''=\frac{1}{2}\left[\oint_{\beta_{j}}\nuI_{i}\right],
$$
Let $\Lambda$ be the lattice in $\CC^g$
spanned by the column vectors of  $2\omega'$  and  $2\omega''$.
The Jacobian variety of  $X_g$  is denoted by  $\mathcal{J}_g$  
and is identified with  $\Bbb C^g/\Lambda$.   
For a non-negative integer $k$,
we define the Abel map from the $k$-th symmetric product $\Sym^k X_g$
of the curve $X_g$ to  $\CC^g$  by first choosing any (suitable)
path of integration\footnote{The results presented below are independent
of the particular choice.}:
\begin{gather*}
\Amp: \Sym^k X_g \to \CC^g, \quad
  \Amp((x_1,y_1), \ldots, (x_k, y_k))= \sum_{i=1}^k
       \int_\infty^{(x_i,y_i)} 
          \begin{pmatrix} \nuI_1\\ \vdots \\\nuI_g \end{pmatrix}.
\end{gather*}
By letting the map $\kappa$ be the natural projection,
\[
\kappa:\mathbb{C}^g \longrightarrow \mathcal{J}_g,
\]
the image of $\kappa \circ\Amp$ is denoted by $\WW_k = 
\kappa \circ \Amp(\Sym^k X_g)$.
%
%
The mapping  $\kappa \circ\Amp$  
is surjective when $k\ge g$ by Abel's theorem, 
and is injective when $k=g$ if we restrict the map to
the pre-image of the complement of a 
specific connected Zariski-closed subset of 
dimension at most $g-2$ in  $\mathcal{J}_g$,  
by Jacobi's theorem.

\subsection{Sigma function and its derivatives}
We define differentials of the second kind,
$$
     \nuII_{j}=\dfrac{1}{2y}\sum_{k=j}^{2g-j}(k+1-j)
      \lambda_{k+1+j} x^kdx,
     \quad (j=1, \ldots, g)
$$
and (half of)
complete hyperelliptic integrals of the second kind
$$      \eta'=\frac{1}{2}\left[\oint_{\alpha_{j}}\nuII_{i}\right], \quad
         \eta''=\frac{1}{2}\left[\oint_{\beta_{j}}\nuII_{i}\right] .
$$
For this basis of a $2g$-dimensional space of meromorphic differentials, 
the half-periods $\omega',\omega'',\eta',\eta''$  satisfy   
the generalized Legendre relation
\begin{equation} 
{\mathfrak M}\left(\begin{array}{cc}0&-1_g\\1_g&0\end{array}\right) 
{\mathfrak M}^T= 
\frac{\imath\pi}{2}\left(\begin{array}{cc}0&-1_g\\1_g&0\end{array}\right). 
\end{equation} 
where  $\mathfrak 
M=\left( 
\begin{array}{cc}\omega'&\omega''\\\eta'&\eta''\end{array} 
\right)$.
Let  $\mathbb{T}={\omega'}^{-1}\omega''$.  
The theta function on  $\mathbb C^g$  with ``modulus''  $\mathbb{T}$ 
and characteristics $\mathbb{T}a+b$ for $a,b\in \CC^g$ is given by
$$
    \theta\negthinspace\left[\begin{matrix}a \\ b \end{matrix}\right]
     (z; \mathbb T)
    =\sum_{n \in \mathbb Z^g} \exp \left[2\pi i\left\{
    \dfrac 12 \ ^t\,\negthinspace (n+a)\mathbb T(n+a)
    + \ ^t\,\negthinspace (n+a)(z+b)\right\}\right].
$$ 
The $\sigma$-function (\cite{B1}, p.336, \cite{BEL}), 
an analytic function on the space  $\CC^g$  and a theta series having 
modular invariance 
of a given weight with respect to  $\mathfrak{M}$, 
is given by the formula
$$
 \sigma(u)
  =\gamma_0\, \mathrm{exp}\left\{-\dfrac{1}{2}\ ^t \ft
  \eta'{{\omega}'}^{-1}\ft   \right\}
  \theta\negthinspace
  \left[\begin{matrix}\delta'' \\ \delta' \end{matrix}\right]
  \left(\frac{1}{2}{{\omega}'}^{-1}\ft \,;\, \mathbb T\right),
$$
where $\delta'$ and $\delta''$ are half-integer characteristics 
giving the vector of Riemann constants with basepoint at $\infty$
   and $\gamma_0$ is a 
certain non-zero constant. 
The $\sigma$-function vanishes exactly on  $\kappa^{-1}(\WW_{g-1})$ (see for
example \cite[XI.206]{B1}). 
The Kleinian $\wp$ and $\zeta$ functions are defined by
$$
\wp_{i j} = -\frac{\partial^2}{\partial u_i \partial u_j} \log \sigma(u),
\quad
\zeta_{i} = \frac{\partial}{\partial u_i} \log \sigma(u).
$$

Let $\{\phi_i(x,y)\}$ be
an ordered  set of $\CC\cup\{\infty\}$-valued functions
 over $X_g$ defined by
\begin{gather}\label{eq:phi}
	\phi_i(x,y) = \left\{
        \begin{array} {llll}
     x^i                 \quad & \mbox{for } i \le g, \\
     x^{\LF(i-g)/2\RF+ g} \quad& \mbox{for } i > 
g~{\rm and}~ i - g\mbox{  even,}\\
     x^{\LF(i-g)/2\RF} y \quad & \mbox{for } i > 
g~{\rm  and}~ i - g\mbox{  odd.}\\
        \end{array} \right.
\end{gather}
We note that $\{\phi_i(x,y)\}$ give a basis of
the (infinite-dimensional)  $\CC$ vector space
$R_g$.

Following \cite{O,MP10},
we introduce a multi-index  $\natural^n$.  
For  $n$  with  $1\leq n<g$, we let  $\natural^n$  
be the set of positive integers  $i$  such that  
$n+1\leq i\leq g$  with  $i\equiv n+1$  mod  $2$.   
Namely,  
\begin{align*}
\natural^n=\left\{\begin{array}{lll}
n+1, n+3, \ldots, g-1 \quad& \text{for}\; g-n\equiv 0\; \mathrm{mod}\; 2\\
n+1,   n+3, \ldots, g   \quad& \text{for}\; g-n\equiv 1\; \mathrm{mod}\; 2\\
 \end{array} \right.
\end{align*}
and partial derivatives over the multi-index $\natural^n$   
$$
	\sigma_{\natural^n}
        =\bigg(\prod_{i\in \natural^n}
	\frac{\partial}{\partial u_i}\bigg)  \sigma(u).
$$
For  $n\geq g$, we define  $\natural^n$  as empty and  
$\sigma_{\natural^n}$  as  
$\sigma$  itself. 
The first few examples are given in Table 1, where we let 
$\sharp$ denote 
$\natural^1$  and  $\flat$ denote $ \natural^2$.

\medskip
\centerline{Table 1}
\smallskip
  \centerline{
  \vbox{\offinterlineskip
    \baselineskip =5pt
    \tabskip = 1em
    \halign{&\hfil#\hfil \cr
     \noalign{\vskip 2pt}
     \noalign{\hrule height0.8pt}
& genus & \hfil\strut{\vrule depth 4pt}\hfil  & 
$\sigma_\sharp\equiv \sigma_{\natural^1}$ & 
$\sigma_\flat\equiv\sigma_{\natural^2}$ & 
$\sigma_{\natural^3}$  & $\sigma_{\natural^4}$ 
& $\sigma_{\natural^5}$ &$\sigma_{\natural^6}$&
$\sigma_{\natural^7}$ &$\sigma_{\natural^8}$ &$\cdots$\cr
      \noalign{\hrule height0.3pt}
& $1$ & \strut\vrule  & $\sigma$        & $\sigma$       & $\sigma$    
   & $\sigma$      & $\sigma$      & $\sigma$   & $\sigma$  
 & $\sigma$ & $\cdots$ \cr
& $2$ & \strut\vrule  & $\sigma_2$      & $\sigma$       & $\sigma$ 
      & $\sigma$      & $\sigma$      & $\sigma$   & $\sigma$  
 & $\sigma$ & $\cdots$ \cr
& $3$ & \strut\vrule  & $\sigma_2$      & $\sigma_3$     & $\sigma$ 
      & $\sigma$      & $\sigma$      & $\sigma$   & $\sigma$  
 & $\sigma$ & $\cdots$ \cr
& $4$ & \strut\vrule  & $\sigma_{24}$   & $\sigma_3$     & $\sigma_4$  
   & $\sigma$      & $\sigma$      & $\sigma$   & $\sigma$   & $\sigma$
 & $\cdots$ \cr
& $5$ & \strut\vrule  & $\sigma_{24}$   & $\sigma_{35}$  & $\sigma_4$  
   & $\sigma_5$    & $\sigma$      & $\sigma$   & $\sigma$   & $\sigma$
 & $\cdots$ \cr
& $6$ & \strut\vrule  & $\sigma_{246}$  & $\sigma_{35}$  & $\sigma_{46}$
  & $\sigma_5$    & $\sigma_6$    & $\sigma$   & $\sigma$   & $\sigma$ &
 $\cdots$ \cr
& $7$ & \strut\vrule  & $\sigma_{246}$  & $\sigma_{357}$ & $\sigma_{46}$ 
 & $\sigma_{57}$ & $\sigma_6$    & $\sigma_7$ & $\sigma$   & $\sigma$ & 
$\cdots$ \cr
& $8$ & \strut\vrule  & $\sigma_{2468}$ & $\sigma_{357}$ & $\sigma_{468}$
 & $\sigma_{57}$ & $\sigma_{68}$ & $\sigma_7$ & $\sigma_8$ & $\sigma$ &
 $\cdots$ \cr
& $\vdots$ & \hfil\strut {\vrule depth 5pt}\hfil  & $\vdots$ & $\vdots$ 
& $\vdots$ & $\vdots$ & $\vdots$  & $\vdots$   & $\vdots$   & $\vdots$ 
& $\ddots$ \cr
	\noalign{\hrule height0.8pt}
}}
}
\medskip

 For  $u\in\CC^g$, we denote by  $u'$  and  $u''$  
the unique vectors in  $\mathbb{R}^g$  such that 
  $$
  u=2\,{}^t\omega'u'+2\,{}^t\omega'' u''.
  $$
We define 
  \begin{equation*}
  \begin{aligned}
  L(u,v)&={}^t{u}(2\,{}^t\eta'v'+2\,{}^t\eta''v''), \\ 
  \chi(\ell)&=
  \exp\big\{2\pi i\big({}^t{\ell'}\delta''-{}^t{\ell''}\delta'
  +\tfrac12{}^t{\ell'}\ell''\big)\big\}
  \ (\in \{1, -1\})
  \end{aligned}
  \end{equation*}
for  $u$, $v\in\CC^g$  and for  $\ell$ 
($=2\,{}^t\omega'\ell'+2\,{}^t\omega''\ell''$) $\in\Lambda$.  
Then  $\sigma_{\natural^n}(u)$  for  $u\in\kappa^{-1}(\WW_1)$ 
 satisfies the translational relation
(\cite{O}, Lemma 7.3):
\begin{equation}
 \sigma_{\natural^n}(u+\ell)
   =\chi(\ell)\sigma_{\natural^n}(u)\exp L(u+\tfrac12\ell,\ell)\ \ 
 \hbox{for $u\in\kappa^{-1}(\WW_1)$}. 
\label{eq:trans}
\end{equation}
Further for  $n\leqq g$,  we note that 
$\sigma_{\natural^n}(-u)=(-1)^{ng+\frac12n(n-1)}\sigma_{\natural^n}(u)$ 
for $u\in \kappa^{-1}(\WW_n)$, especially, 
\begin{equation}
\left\{
\begin{array}{llll}
\sigma_{\flat}(-u)=-\sigma_{\flat}(u) \ \ 
         & \hbox{for $u\in \kappa^{-1}(\WW_2)$} \\[1.5ex]
\sigma_{\sharp}(-u)=(-1)^g\sigma_{\sharp}(u) \ \ 
        &\hbox{for $u\in \kappa^{-1}(\WW_1)$}  
\end{array}\right.
\label{eq:sign}
\end{equation}
by Proposition 7.5 in \cite{O}.

\section{The addition formulae}
In this Section, we give the addition formulae of the hyperelliptic 
$\sigma$ functions which are the generalization
of the addition formula (\ref{eq:add1}). 
These are essential  to construct the hyperelliptic
solution of the Toda lattice.

\subsection{Generalized Frobenius-Stickelberger formula}
We first recall the generalized Frobenius-Stickelberger formula
which gives a generalization of the addition formula (\ref{eq:add1}).

\begin{definition}{\rm 
For a positive integer $n\geq 1$ and
 $(x_1,y_1), \ldots, (x_n, y_n)$ in $X_g$,
we define the Frobenius-Stickelberger determinant \cite{MP},
\begin{gather*}
\begin{split}
&\Psi_n((x_1, y_1), \ldots, (x_n, y_n))\\
& := 
\left|
\begin{matrix}
1 & \phi_1(x_{1}, y_{1}) &\cdots & \phi_{n-2}(x_{1}, y_{1}) & \phi_{n-1}
(x_{1}, y_{1})  \\
1 & \phi_1(x_{2}, y_{2}) & \cdots & \phi_{n-2}(x_{2}, y_{2}) & \phi_{n-1}
(x_{2}, y_{2})  \\
\vdots & \vdots & \ddots & \vdots & \vdots \\ 
1 & \phi_1(x_{n-1}, y_{n-1})  &\cdots & \phi_{n-2}(x_{n-1}, y_{n-1}) & 
\phi_{n-1}(x_{n-1}, y_{n-1})  \\
1 & \phi_1(x_{n}, y_{n})  & \cdots & \phi_{n-2}(x_{n}, y_{n}) & \phi_{n-1}
(x_{n}, y_{n})  \\
\end{matrix}
\right|.
\end{split}
\end{gather*}}
where $\phi_i(x_j,y_j)$'s are the monomials defined in (\ref{eq:phi}).
\end{definition} 
Then we have the following theorem (Theorem 8.2 in \cite{O}):
\begin{proposition} \label{prop:FSform}
For a positive integer $n>1$,
let $(x_1,y_1), \ldots, (x_n, y_n)$  in  $X_g$, 
and $u^{(1)}, \ldots, u^{(n)}$  in  $\kappa^{-1}(\WW_1)$  be points such
 that  
$u^{(i)}= w((x_i, y_i))$.
Then the following relation holds\,{\rm :}
\begin{gather}
\frac{\sigma_{\natural^n}(\sum_{i=1}^{n} u^{(i)})
\prod_{i<j}\sigma_{\flat}(u^{(i)} - u^{(j)}) }
{\prod_{i=1}^n\sigma_{\sharp}(u^{(i)})^n}
=\epsilon_n \Psi_n((x_1, y_1), \ldots, (x_n, y_n)),
\label{eq:FS0}
\end{gather}
where $\epsilon_n=(-1)^{g+n(n+1)/2}$ for $n\leq g$
and $\epsilon_n=(-1)^{(2n-g)(g-1)/2}$ for $n\geq g+1$.
\end{proposition}

\subsection{The algebraic addition formula} 

We first describe linear equivalence
of divisors on the curve $X_g$ (which will result on
the addition law on the Jacobian) by  algebraic formulas
 \cite{MP}:
\begin{definition}
For given $P_1, \ldots, P_n \in X_g$,
we define
$$
\mu_n(P; P_1, \ldots, P_n) =
\lim_{Q_i \to P_i} \frac{\Psi_{n+1}(P, Q_1, \ldots, Q_n)}{\Psi_n(Q_1, \ldots,
 Q_n)}
$$
for distinct $Q_i$'s (the order in which the limits are taken is irrelevant).
\end{definition}

\begin{proposition}\label{prop:4.4}
For given $P_1, \ldots, P_n \in X_g$, we find
 $Q_1, \ldots, Q_\ell$  with
$\ell = g$ for $n\ge g$ and $\ell = n$ otherwise, such that
$$
	P_1 + P_2 +  \cdots + P_n + Q_1 + Q_2 + 
 \cdots + Q_\ell - (n+\ell) \infty \sim 0
$$
by taking the zero-divisor of $\mu_n(P; P_1, \ldots, P_n)$.
For each $Q_i = (x_i, y_i)$, by letting $-Q_i = (x_i, -y_i)$,
we have the addition property,
$$
	P_1 + P_2 +  \cdots + P_n - n\infty \sim 
 (-Q_1) + (-Q_2) + \cdots + (-Q_\ell) - (-\ell) \infty.
$$
\end{proposition}

As usual, the symbol of  addition in the (free abelian) 
divisor group is used as well
for addition up to linear equivalence.
\begin{remark}
{\rm{
We should note that the hyperelliptic involution 
$\iota: (x,y) \mapsto (x, -y)$ induces the $[-1]$-action on
$\JJ_g$, defined by $u \mapsto -u$.
}}
\end{remark}

\subsection{The analytic addition formula} 
We have the following 
addition formula for the hyperelliptic $\sigma$ functions 
(Theorem 5.1 in \cite{EEMOP}):
\begin{theorem}\label{thm:add} 
Assume that $(m, n)$ is a pair of positive integers. 
Let  $(x_i,y_i)$  $(i=1, \ldots, m)$, $(x'_j,y'_j)$ $(j=1, \ldots, n)$  in  
$X_g$ 
and  $u\in\kappa^{-1}(\WW_m)$, $v\in\kappa^{-1}(\WW_n)$  be points 
such that  $u=w((x_1, y_1),\ldots,(x_m,y_m))$ 
and  $v=w((x_1', y_1'),\ldots,(x_n',y_n'))$.  
Then the following relation holds\,{\rm :}
\begin{gather}
\begin{split}
&\frac{\sigma_{\natural^{m+n}}(u + v) \sigma_{\natural^{m+n}}(u - v) }
{\sigma_{\natural^m}(u)^2 \sigma_{\natural^n}(v)^2} \\
&=\delta(g,m,n)
\frac{\prod_{i=0}^1 \Psi_{m+n}
((x_1,y_1),\ldots, (x_m, y_m), 
(x_1',(-1)^i y_1'),\ldots,(x_n',(-1)^i y_n'))}
{(\Psi_{m}((x_1,y_1),\ldots, (x_m, y_m))
\Psi_{n}((x_1',y_1'),\ldots, (x_n', y_n')))^2}\\
&\times
\prod_{i=1}^m\prod_{j=1}^n 
\frac{1}{ \Psi_2((x_i, y_i), (x_j', y'_j))}
\label{eq:Thadd}
\end{split}
\end{gather}
where $\delta(g,m,n)=(-1)^{gn+\frac12n(n-1)+mn}$.
\end{theorem}

Theorem \ref{thm:add} with $m=g$ and $n=2$ leads to the following Corollary:
\begin{corollary}\label{cor:add} 
Let  $(x_i,y_i) \in X_g$  $(i=1, \ldots, g)$, $(x'_j,y'_j)\in X_g$ $(j=1, 2)$,
 $u\in \CC^g$, 
$v := v^{[1]} +v^{[2]}\in\kappa^{-1}(\WW_2)$, and 
$v^{[j]}\in\kappa^{-1}(\WW_1)$ $(j=1,2)$ be points 
such that  $u=w((x_1, y_1),\ldots,(x_g,y_g))$ 
and  $v^{[j]}=w((x_j', y_j'))$, $(j=1,2)$.  
Then the following relation holds\,{\rm :}
\begin{gather}
\frac{\sigma(u + v) \sigma(u - v) }
{\sigma(u)^2 \sigma_{\flat}(v)^2}= - \Xi(u, v),
\label{eq:add_n_2}
\end{gather}
where $\Xi(u, v)$ is equal to
$$
F(x_1') F(x_2')
\left(\sum_{i=1}^g \frac{y_i}{ (x_i-x_1') (x_i-x_2') F'(x_i)}\right)^2
-
F(x_1') F(x_2')
\left(\sum_{i=1}^2 \frac{(-1)^iy'_i}{ (x_1'-x_{2}')F(x_{i}')}\right)^2,
$$
and $ F(x) := (x - x_1) (x - x_2) \cdots (x - x_g)
\equiv \mu_g((x,y);(x_1, y_1), \ldots, (x_g,y_g))$ and
 $F'(x) := \partial F(x)/ \partial x$.
\end{corollary} 

\begin{proof} 
By letting $\Delta(x_1,x_2,\ldots,x_\ell)$ be the Vandermonde
determinant, i.e.,
\begin{equation}
\begin{split}
\Delta(x_1,x_2,\ldots,x_{\ell})&=\left|\begin{matrix}
1 & x_1 &  \cdots &  x_1^{\ell-1} \\ 
1 & x_2 &  \cdots &  x_2^{\ell-1} \\ 
\vdots&  \vdots& \ddots& \vdots\\
1 & x_{\ell} &  \cdots &  x_{\ell}^{\ell-1} \\ 
\end{matrix}\right| =\prod_{i, j  = 1, i<j}^\ell (x_j - x_i),
\end{split}
\end{equation}
we have
\begin{equation}
\begin{split}
&\Psi_{g+2}((x_1, y_1),  \ldots ,(x_g, y_g), 
(x'_1, \pm y'_1), (x'_2, \pm y'_2))
=\left|\begin{matrix}
1 & x_1 & \cdots &  x_1^{g} & y_1\\ 
1& x_2 & \cdots & x_2^g & y_2\\
\vdots& \vdots&  \ddots& \vdots& \vdots \\
1 & x_{g} &  \cdots & x_{g}^{g} & y_g\\ 
1 & {x'}_{1} & \cdots &{x'}_{1}^{g}&\pm{y'}_{1}\\ 
1 & {x'}_{2} & \cdots &{x'}_{2}^{g}&\pm{y'}_{2} \\ 
\end{matrix}\right|\\
&=
\sum_{i=1}^g
\Delta(x_1, \ldots,\check{x_i},\ldots, x_g, x_1', x_2') y_i 
\pm \sum_{i=1}^2
(-1)^{i}\Delta(x_1,  \ldots,x_g, x_{3-i}') y'_i .
\end{split}
\end{equation}
Thus for the case $m=g$ and $n=2$,
the right-hand side of the formula in
 Theorem \ref{thm:add} is equal to
(\ref{eq:add_n_2}).
\end{proof} 

Baker \cite[\S11.217]{B1}, \cite[p. 138]{B3} proves the following:
\begin{proposition}\label{prop:baker}
Let  $(x_i,y_i) \in X_g$  $(i=1, \ldots, g)$ and
 $u\in \CC^g$ such that  $u = $ $w((x_1, y_1),\ldots,(x_g,y_g))$.
The following relation holds for
generic $x_i'$ $(i=1,2)$, 
\begin{gather}
\begin{split}
\sum_{i=1}^g \sum_{j=1}^g \wp_{i j}(u)
 {x'_1}^{i-1} {x'_2}^{j-1}&=
F(x'_1)F(x'_2)\left(\sum_{i=1}^g
\frac{y_i}
  {(x'_1-x_i)(x'_2-x_i)F'(x_i)}\right)^2\\
&-\frac{f(x'_1)F(x'_2)}{(x'_1-x'_2)^2F(x'_1)}
-\frac{f(x'_2)F(x'_1)}{(x'_1-x'_2)^2F(x'_2)}
+\frac{f(x'_1,x'_2)}{(x'_1-x'_2)^2},
\label{eq:prop:baker}
\end{split}
\end{gather}
where
$$
 f(x_1, x_2) = \sum_{i=0}^g x_1^i x_2^i
(\lambda_{2i+1} (x_1 + x_2) + 2 \lambda_{2i}).
$$
\end{proposition}

Corollary \ref{cor:add} and Proposition \ref{prop:baker} yield
the key proposition in this article.

\begin{proposition}\label{prop:fay}
For the variables in Corollary \ref{cor:add},
 the following relation holds\,{\rm :}
\begin{gather}
\frac{\sigma(u + v) \sigma(u - v) }
{\sigma(u)^2 \sigma_{\flat}(v)^2}  
= \frac{f(x_1', x_2')-2 y_1' y_2' } {(x'_1-x'_2)^2} -
\sum_{i=1}^g \sum_{j=1}^g \wp_{i j}(u)
 {x'_1}^{i-1} {x'_2}^{j-1}.
\label{eq:prop:fay}
\end{gather}
\end{proposition}
This corresponds to Fay's formula \cite[(39)]{F},
which is the basis of ``Fay's trisecant identity''.

\begin{remark} \label{rmk:fay_g1}
{\rm{
For the $g=1$ case,
$\displaystyle{\frac{f(x_1', x_2')-2 y_1' y_2' } {(x'_1-x'_2)^2}
= \wp(v_1 + v_2)}$. Thus if we let $v_3 = v_1 + v_2$ and 
$x'_3=\wp(v_3)$,
(\ref{eq:prop:fay}) recovers
(\ref{eq:add1}).
}}
\end{remark}

\begin{corollary} \label{cor:deg1}
For the variables in Corollary \ref{cor:add},
with $v^{[1]} =v^{[2]}$, we have 
\begin{gather}
\begin{split}
\frac{\sigma(u + 2 v^{[1]}) \sigma(u -2 v^{[1]}) }
{\sigma(u)^2 \sigma_{\flat}(2v^{[1]})^2} 
&= f_{1,2}(x_1') -
 \sum_{i=1}^g \sum_{j=1}^g \wp_{i j}(u)
 {x'_1}^{i + j - 2}\\
&= -\lim_{x_2' \to x_1'} \Xi(u, v).
\end{split}
\end{gather}
where
$$
f_{1,2}(x):=
\frac{\partial_{x}^2f(x)}{2f(x)} - f_{1,2}^I(x), \quad
f_{1,2}^I(x):=\sum_{i=0}^g 
(i^2\lambda_{2i+1} x^{2i-1}+ i(i-1) \lambda_{2i}x^{2i}).
$$
\end{corollary}

Belokolos, Enolskii, and Salerno
gave the following relation \cite[Theorem 3.2]{BES}:
\begin{corollary} \label{cor:deg2}
For the variables in Corollary \ref{cor:add},
with $v^{[2]} = 0$ or $(x_2',y_2')=\infty$, we have 
\begin{gather}
\begin{split}
\frac{\sigma(u + v^{[1]}) \sigma(u - v^{[1]}) }
{\sigma(u)^2 \sigma_{\flat}(v^{[1]})^2} 
 &={x'_1}^g - \sum_{i =1}^g \wp_{g j}(u)
 {x'_1}^{i + j -2}\\
\frac{\sigma(u + v^{[1]}) \sigma(u - v^{[1]}) }
  {\sigma(u)^2 \sigma_{\flat}(v^{[1]})^2} 
&={x'_1}^g - \sum_{i =1}^g \wp_{g j}(u)
   {x'_1}^{i -1}\\
 &=F(x'_1) = (x'_1 - x_1) (x'_1 - x_2) \cdots (x'_1 - x_g)\\
&\equiv \mu_g((x_1',y_1');(x_1, y_1), \ldots, (x_g,y_g)).
\end{split}
\end{gather}
\end{corollary}

\begin{proof}
We divide both sides of (\ref{eq:prop:fay})
by ${x_2'}^{g-1}$. By taking the appropriate limit,
 we obtain the equality.
\end{proof}

\begin{remark}\label{RevSigma}
{\rm{
We conclude this Section by elaborating on the geometric
properties of the $\sigma$ function which make the addition formulas
particularly suited for integrating equations of dynamics.
As stated in the Introduction, 
 originally the sigma function was 
discovered by Weierstrass \cite{W1} 
in order to express a symmetric function of the points of 
a hyperelliptic curve, which he
called ``al'' function, in terms of rational functions.
 The al function is a root function, 
equal to $\sqrt{F(b_a)}$ for a zero $b_a$ of $f(x)$, giving one of the branch
points of $X_g$.
The focus was on constructing a dictionary 
between  abelian and rational functions. In particular, 
a goal was the solution of the ``Jacobi inversion'',
namely returning the symmetric functions of the divisor
from a point on the universal cover of the Jacobian,
as in the genus-2 case:
$\wp_{22} = x_1 + x_2$, $\wp_{21} = x_1 x_2$, where the points 
 of the divisor $P_1+P_2$ are $P_i=(x_i,y_i),\ i=1,2$.
(see also Remark \ref{rmk:BakerFayToda}(2)(b)).
The most natural application is then the explicit realization of the
group structure of the Jacobian
(the generalized Frobenius-Stickelberger relation in Proposition
 \ref{prop:FSform}, which  gives the addition structure,
shows a simple connection between the affine coordinate ring $R_g$
and the group law on the Jacobian $\JJ_g$), 
reflecting addition in the free abelian
divisor group, and this is achieved in 
Corollary \ref{cor:deg1}, vis-\`a-vis Proposition \ref{prop:4.4}.}
}
\end{remark}

\section{The $\sigma$-function solution of the Toda lattice}\label{Solution}

In this Section, we give the $\sigma$-function solution of the Toda 
lattice equation.
\medskip

First, we identify some algebraic identities that hold for vector
fields on the Jacobian. Vector fields on the Jacobian
are understood to be translation invariant; equivalently, they are
elements of the tangent space at the origin. It is important
that we use algebraic functions on the curve to write their
coordinates in the abelian variables $(u_i)$, but
in doing this, we make the convention that we are on a suitable
coordinate patch on the Jacobian, where the Abel map from 
$\mathrm{S}^gX_g$ can be inverted; as explained in \cite[\S3]{Mu},
there are choices involved
and one has to check that  the vector field in question is well defined.
Our formulas would hold on this suitable affine patch anyway, 
since outside it,
the Toda orbits become of smaller dimension, as mentioned
in the Introduction; that case will be addressed in our forthcoming work.
Moreover, as usual,
we view the vector fields as derivations on the universal cover,
namely on the ring of functions  
 in $\Gamma (\mathbb{C}^g,\mathcal{O}(\mathbb{C}^g))$.

\bigskip

In order to give our solution of the Toda lattice equation,
we need only two vector fields on the Jacobian, each associated in the
same way to a fixed point  $(x_j',y_j') \in X_g$,
$j=1, 2$; however, in order to relate our solution
 to Baker's differential
equations  (Remark 
\ref{rmk:BakerFayToda}), we distinguish two further fixed points
$(x_j',y_j') \in X_g$,  $j=3, 4$, for which we use
Baker's expression of the derivatives  in terms of algebraic
functions on the curve as opposed to abelian coordinates
on the Jacobian; the dictionary between the two expressions
is given in Lemma \ref{lm:diff} (b). 
\begin{definition}\label{algebraic}
For $(x_i,y_i) \in X_g$  $(i=1, \ldots, g)$,
 $u=w((x_1, y_1),\ldots,(x_g,y_g))\in \CC^g$,
and $(x_j',y_j') \in X_g$ $(j=1, \ldots, 4)$,
we let
$$
D_j:=\sum_{i=1}^g {x_j'}^{i-1} \frac{\partial}{\partial u_i}.
$$
\end{definition}

\begin{lemma}\label{lm:diff}
Let  $(x_i,y_i) \in X_g$  $(i=1, \ldots, g)$, $(x'_j,y'_j)\in X_g$ 
$(j=1, \ldots, 4)$,  $u\in \CC^g$,  and 
$v^{[j]}\in\kappa^{-1}(\WW_1)$ $(j=1, \ldots, 4)$ be points 
such that  $u=w((x_1, y_1),\ldots,(x_g,y_g))$ 
and  $v^{[j]}=w((x_j', y_j'))$, $(j=1, \ldots, 4)$.
We have the following expressions:
\begin{enumerate}
\item[{\rm (a)}]  For each $j=1, 2, 3,$ or $4$, 
\begin{align*}
D_j&=
\frac{1}{\Delta (x_1,x_2,\ldots,x_g)}\left|\begin{matrix}
1 & x_1 & \cdots & x_1^{g-1} &  2 y_1 \partial_{x_1} \\
1 & x_2 &  \cdots & x_2^{g-1} &  2 y_2 \partial_{x_2}\\
\vdots& \vdots &  \ddots& \vdots &\vdots\\
1 & x_g &  \cdots & x_g^{g-1} &  2 y_g \partial_{x_g}\\
1 & x_j' &  \cdots & {x_j'}^{g-1} & 0
\end{matrix}\right|\\
&=\sum_{i=1}^g\frac{y_i F(x_j')}{F'(x_i)(x_j'-x_i)}
\frac{\partial}{\partial x_i},
\end{align*}
and for each $j, j'=1, \ldots, 4$, 
$$
[D_j, D_{j'}]= D_jD_{j'}-D_{j'} D_{j} =0.
$$
\item[{\rm (b)}] For $v^{(j)}=w(x_j',y_j')$ with $j=1, \ldots, 4$,
$$
\frac{\partial}{\partial x_j'}
= \frac{1}{2  y_j'}\sum_{i=1}^g 
{{x_1'}^{i-1}}
\frac{\partial}{\partial v^{(j)_i}}.
$$
\item[{\rm (c)}] For $h \in \Gamma(\CC^g, \cO(\CC^g))$ and $j=1, \ldots, 4$,
$$
D_{j'}
h(u+v^{(j)}):=2  y_j'\frac{\partial}{\partial x_j'}h(u+v^{(j)})
=D_j h(u+v^{(j)}).
$$
\end{enumerate}
\end{lemma}
\begin{proof}
A direct calculation gives the results.
\end{proof}

\begin{lemma} \label{lm:Dlogsigma}
For the variables in Corollary \ref{cor:add}, we have
\begin{enumerate}
\item[{\rm (a)}]
$$
D_1 \log \sigma(u + v) 
=
\frac{1}{2}\left( D_1 \log \Xi(u,v) + D_{1'} \log \Xi(u,v)\right) 
 +D_1 \log \sigma(u) 
 +D_{1'} \log \sigma_{\flat}(v) ,
$$
\item[{\rm (b)}]
\begin{align*}
D_1D_2 \log \sigma(u + v) 
&= \frac{1}{2}D_1D_2 \log \Xi(u,v) +\frac{1}{2}D_1D_{2'} \log \Xi(u,v) 
 +D_1D_2 \log \sigma(u) \\
&= \frac{1}{2}D_1D_2 \log \Xi(u,v) +\frac{1}{2}D_2D_{1'} \log \Xi(u,v) 
 +D_1D_2 \log \sigma(u) .
\end{align*}
\end{enumerate}
\end{lemma}
\begin{proof}
By taking the derivatives of the logarithm of (\ref{eq:add_n_2}), we obtain
\begin{align*}
D_1 \log \Xi(u,v) &= D_1 \log \sigma(u + v) 
 -2D_1 \log \sigma(u) + D_1 \log \sigma(u - v), \\
D_{1'} \log \Xi(u,v) &= D_{1'} \log \sigma(u + v) 
 -2D_{1'} \log \sigma_{\flat}(v) + D_{1'} \log \sigma(u - v),
\end{align*}
and  Lemma \ref{lm:diff} (c) yields the claims.
\end{proof}

Now we can give the sigma-function solution of the Toda lattice equation:
\begin{theorem} \label{thm:Toda}
Let  $(x_i,y_i) \in X_g$  $(i=1, \ldots, g)$, $(x'_1,y'_1)\in X_g$ 
 $u\in \CC^g$,  and $v^{[1]}\in\kappa^{-1}(\WW_1)$ be points 
such that  $u=w((x_1, y_1),\ldots,(x_g,y_g))$ 
and  $v^{[1]}=w((x_1', y_1'))$.  
We define $c:=2 v^{[1]}$,  $\tilde D_1 = \sigma_{\flat}(c) D_1$, 
$$
\cV(u):=
\cV(u,v^{[1]}):=
 \sum_{i =1}^g \sum_{ j =1}^g \wp_{i j}(u)
 {x'_1}^{i + j -2},
\quad
\cV_c(c):= f_{1,2}(x_1'), 
$$
and $t:=(t_{11},t_{12},\ldots, t_{1g})\in\mathbb{C}^g$ with
\[
t_{1j}:=(x_1')^{1-j}\sum_{i=1}^g \dint_{\infty}^{(x_i,y_i)}\nuI_{j},\qquad
 (j=1,2,\ldots,g).
\]
Then with the coordinate change $u=nc+t^\perp + t$ (which defines
$t^\perp$), we have
\begin{enumerate}
\item
\begin{equation}
\begin{split}
   -D_1^2& \log \left(\cV(t +nc + t^\perp) - \cV_c(c) \right) \\
&= \cV(t + (n+1) c +  t^\perp )
	 -2\cV(t + nc + t^\perp ) + \cV( t  + (n-1) c + t^\perp).
\end{split}
\end{equation}
\item The Hirota bilinear equation,
\begin{equation}
\begin{split}
	\sigma(t + nc + t^\perp) 
	\tilde D_1^2 \sigma(t +nc + t^\perp) 
	&-\tilde D_1 \sigma(t +nc + t^\perp)
 \tilde D_1 \sigma(t +nc + t^\perp)\\ 
-\cV_c(c) \sigma(t + n c +  t^\perp )^2
&- \sigma(t + (n+1) c +  t^\perp )\sigma(t + (n-1) c +  t^\perp ) = 0.
\end{split}
\end{equation}
\item The Toda-lattice equation, 
by letting $\cV_n(t + t^\perp) := -\cV(t +nc + t^\perp)$
and $q_n(t) :=- \log \left(\cV_n(t + t^\perp) - \cV_c(c) \right)$,
\begin{equation}
\begin{split}
	-D_1^2& q_n(t) = \ee^{-q_{n+1}} - 2\ee^{-q_{n}} + \ee^{-q_{n-1}}.
\end{split}
\end{equation}
\end{enumerate}
\end{theorem}

\begin{proof} The claims follow 
from Corollary \ref{cor:deg1}, by direct verification, as in the genus-1
case, using
the definition of the vector fields.
\end{proof}

We  also have the sigma function solutions of
the two-(time-)dimensional Toda lattice equation 
\cite[3.5.1, (3.107-108)]{Hi},
$\partial^2 q_n/\partial t_1\partial t_2=V_{n+1}-2V_n+V_{n-1},
\ q_n=\log (1+V_n),$ $n$ any integer:

\begin{theorem} \label{thm:twoToda}
Let  $(x_i,y_i) \in X_g$  $(i=1, \ldots, g)$, $(x'_j,y'_j)\in X_g$ $(j=1, 2)$,
 $u\in \CC^g$, 
$v := v^{[1]} +v^{[2]}\in\kappa^{-1}(\WW_2)$, and 
$v^{[j]}\in\kappa^{-1}(\WW_1)$ $(j=1,2)$ be points 
such that  $u=w((x_1, y_1),\ldots,(x_g,y_g))$ 
and  $v^{[j]}=w((x_j', y_j'))$, $(j=1,2)$.  
We define $c:= v^{[1]} + v^{[2]}$,
$$
\hat \cV(u,v^{[1]},v^{[2]}):=
 \sum_{i = 1}^g \sum_{j =1}^g \wp_{i j}(u)
 {x'_1}^{i -1}{x'_2}^{j -1},
$$
$$
\hat \cV_c(v^{[1]},v^{[2]}):=
 \frac{2 y_1' y_2' - f(x_1', x_2')} {(x'_1-x'_2)^2},
$$
and $t_j:=(t_{j1},t_{j2},\ldots,t_{jg})\in \mathbb{C}^g$ with
\[
t_{jk}:=(x_j')^{1-k}\sum_{i=1}^g \dint_{\infty}^{(x_i,y_i)}\nuI_{k},
\qquad (j=1,2,~{\rm and}~k=1,2,\ldots,g).
\]
Then with $u=nc+t^\perp +t_1 + t_2$,
we have
\begin{gather*}
\begin{split}
           -D_1 D_2& \log 
(\tilde \cV(nc +t_1 + t_2 + t^\perp) - \tilde \cV_c(c)) \\
&= \tilde \cV(t_1 + t_2 + (n+1) c +  t^\perp )
	 -2\tilde \cV(t_1 + t_2 + nc + t^\perp ) 
+ \tilde \cV( t_1 + t_2  + (n-1) c + t^\perp).
\end{split}
\end{gather*}
\end{theorem}

\begin{proof} The claims follow 
from Proposition \ref{prop:fay} by direct verification.
\end{proof}

\begin{remark}\label{Hirota}
{\rm As Hirota notes, a two-(space-)dimensional version of the Toda 
system
ought to have two independent discrete variables, but 
at the time, he had only found solitons for the two-times case
\cite[\S3.5.1, Remark]{Hi}. 
Instead of Theorem \ref{thm:twoToda}, as a different
type of generalization of Theorem \ref{thm:Toda} involving  
the sigma function on the Jacobian,
 implementing a second spatial (discrete) 
variable would entail choosing another  $z^{[1]}\in\mathcal{W}_1$ 
say, and
adding to the argument of the $\sigma$-function a vector
$mc_1+nc_2$, with
$c_1=2v^{[1]}$ and $c_2=2z^{[1]}$,
and ``considering a two-dimensional
version of the force term on the right-hand side'' of
(1) in Theorem
\ref{thm:Toda} [\textit{ibid.}]. We would need a different
kind of addition formula, which  we have not yet 
developed. }
\end{remark}

\begin{remark}  \label{rmk:BakerFayToda}
{\rm{We would like to stress  that a single formula 
underlies the validity of the integrable hierarchies:
the differential operators that Baker used to obtain
what we call the KdV and KP equations are always of the
type $D_1$ we defined, namely a linear combination
of the Abelian coordinates given by the basis
of holomorphic differentials which involve the equation of
the plane curve, against coefficients that are powers
of the $x$-coordinate of one point of the curve;
generically, distinct points yield independent vector fields.
Specific to our situation:
\begin{itemize}
\item[(1)] Corollary \ref{cor:add},
Baker's formula (\ref{eq:prop:baker}) 
in \cite[p. 328]{B1} and \cite[p. 138]{B3},
and 
Fay's formula (\ref{eq:prop:fay}) in \cite[(39) in p. 26]{F}
are essentially the same.
We notice the following two facts:
 
\begin{enumerate}
\item[(a)] Baker derived the KdV hierarchy and KP equation by using 
his formula (\ref{eq:prop:baker}), using the
 vector fields $D_3$ and $D_4$, cf.
Definition \ref{algebraic} \cite{Ma2}.

\item[(b)] Fay derived his famous trisecant identity,
which is equivalent to the KP hierarchy, by
using  formula (\ref{eq:prop:fay}), which is the hyperelliptic
sigma-function version. 
\end{enumerate}

\smallskip
\item[(2)]  The following relationships hold among the formulae:

\begin{enumerate}
\item[(a)] The two-dimensional Toda equation in Theorem \ref{thm:twoToda}
is  the same as formula (\ref{eq:prop:fay}).

\item[(b)] When $v^{[1]} =v^{[2]}$
in Theorem \ref{thm:twoToda}, we obtain  the Toda lattice equation
of Theorem \ref{thm:Toda}, 
and when $v^{[2]}=0$, the Toda equations are obtained by differentiating
$F(x_1^\prime )$  (Corollary 
\ref{cor:deg2}).
The function $F(x_1')$ is a polynomial of degree $g$
in the variable $x_1'$. By applying the vector field 
$D_3$,
 we obtain the ``Mumford triple'' (which is called $(U,V, W)$
in \cite{Mu}, three polynomials that parametrize
the Jacobian outside a theta divisor
(excluded are certain $g$-tuples of points $(x_i,y_i)$ in special position). 
We note that the coefficients of Mumford's $(U,V,W)$
involve only the abelian functions $\wp_{g i}$, $i\ge 0$ \cite[\S10]{Mu}.
The KdV hierarchy follows again using formula (\ref{eq:prop:fay}).
\end{enumerate}
\smallskip
\item[(3)] We stress again the connection between algebraic
and abelian functions:
 the differential operator $D_j$ has an expression
that involves only (and acts upon) the $x_i$'s (cf. 
Lemma \ref{lm:diff}). On the other hand,
(\ref{eq:add_n_2}) is also given by the affine coordinates
of points of $X_g$.
Hence the Toda lattice equation is an identity defined
over $\Sym^g X_g$. It can also be regarded as a relation among the functions
over $\JJ_g$ and  $\Sym^g X_g$.
\end{itemize}
}}
\end{remark}

\begin{remark}\label{comparison} 
{\rm{We conclude by comparing our solution to 
some of the methods that were used to obtain algebro-geometric
Toda flows.
\begin{itemize}
\item[(1)] van Moerbeke \cite{vM}, following his work with Kac on the Toda 
lattice and Jacobi matrices, reported below in Section 
\ref{Periodicity}, gave a description of the isospectral flows
in terms of linear flows on a Jacobian, and with Mumford
\cite{vanMM}, the algebraic coordinates for the flows and
the solutions in terms of theta functions. The flows are
linear combinations of, essentially, the $D_i$ defined above,
for $(x_i',y_i')$ a branchpoint of the hyperelliptic involution.
Fay's trisecant identity and its derivatives along the $D_i$ are used 
\cite[\S5, Proposition]{vanMM} to show that the flows are Hamiltonian
vector fields that preserve the spectrum of the matrices.
\item[(2)] Algebro-geometric
solutions to the Toda lattice can be found in
\cite{GHMT}. These authors in their extensive work also
used the spectral curve of the tridiagonal matrices
whose deformations, equivalent to the Toda lattice, are given below
in Section \ref{Periodicity}. To solve them, using
the divisors given by auxiliary spectra and via eigenvectors
expressed in terms of theta functions, the authors work out a 
``discrete Floquet theory'' (as had done Kac and van Moerbeke)
and solve the ``Dubrovin equations'' on the expansion of the Green's function
by Abelian functions.
\end{itemize}
}}
\end{remark}
\bigskip

\section{Periodicity of the Toda-lattice solution}\label{Periodicity}

In this section we turn to the problem of periodic solutions
of the Toda lattice.  Spatial periodicity is of physical interest,
in the lattice case, so the requirement
amounts to finding a point of finite order $N$ 
on the hyperelliptic curve:
$c = 2w((x'_1, y'_1))$ such that for  $u \in \CC^g$,
$$
u + N c  = u \quad \mbox{modulo} \quad \Lambda.
$$ Hyperelliptic curves that admit such a point are special, and were 
called ``Toda curves'' by McKean and van Moerbeke \cite{McKvM},
who proved that they are dense in moduli space.
For the same reason, points of finite order in the Jacobian that
come from the curve, or from the sum of a specific number of points on the
curve, give periodic orbits in the billiard (completely
integrable Hamiltonian) system in the ellipsoid \cite{DR}, and are related
to Poncelet's closure (cf. Appendix).

The finite-point condition was
investigated by Cantor and \^Onishi \cite{O, C} 
using the division polynomial $\psi_{2N}$, an element of $R_g$.
Similarly, we investigate
the periodic solutions of the Toda lattice.

\subsection{Division polynomials $\psi_{2N}$}
Traditionally, division polynomials $F_n(x,y)$
arise in expressing the coordinates of
$nP$ in terms of those of $P$, a point of an elliptic curve
in Weierstrass form. In particular, given our present interest,
we call ``the division polynomial'' an element of the
ring of functions on the affine curve, whose solutions are
points of finite order in the Jacobian (where the Abel map
is, as usual, based at $\infty$).
The division polynomial for the genus-one 
case was studied by Kiepert \cite{Ki}
and Brioschi \cite{Br}.
We call Kiepert-type polynomial and Brioschi-type polynomial
the genus-one version of the division polynomial 
$\psi_n$, a polynomial whose zeros $P$
satisfy the condition $n P \equiv 0$, more precisely,
$n w(P) \equiv 0$ modulo $\Lambda$.

Referring  to \cite[Definition 9.2]{O},
the hyperelliptic version of the $\psi_n$ function for genus $X_g$ 
over $w(X_g) = \kappa^{-1} \WW_1$ is defined by
$$
	\psi_n(u)=\frac{\sigma_{\natural^n}( nu )}{\sigma_{\sharp}(u)^{n^2}}.
$$
A zero $u(\ne 0)$ of $\psi_n$ has the property that 
$n u \in \kappa^{-1} \WW_{g-1}$.
The transformation law under translation (\ref{eq:trans}) 
allows one to check that
$\psi_n$ belongs to $R_g$.
Thus it is a natural generalization of the classical Kiepert formula, 
or the division polynomial.
By taking limits of Proposition \ref{prop:FSform}
along the Abelian variables, 
we can give an expression for $\psi_n$ in terms of $\phi_i$'s in $R_g$
 \cite[Theorem 9.3]{O}.
In \cite{MP}, we proved the following:
\begin{theorem} \label{th:Kiep}
Let $n\ge 1$ be a positive integer.
For
\begin{gather*}
\psi_{n}(u)= \frac{\varepsilon_{n,g}}{1! 2! \cdots (n-1)!}
\left|
\begin{matrix}
\frac{\partial \phi_1}{\partial \ugo} &
        \frac{\partial \phi_2}{\partial \ugo} &
               \cdots &\frac{\partial \phi_{n-1}}{\partial \ugo}\\
\frac{\partial^2 \phi_1}{\partial \ugo^2} &
        \frac{\partial^2 \phi_2}{\partial \ugo^2} &
               \cdots &\frac{\partial^2 \phi_{n-1}}{\partial \ugo^2}\\
\vdots & \vdots & \ddots& \vdots\\
\frac{\partial^{n-1} \phi_1}{\partial \ugo^{n-1}} &
        \frac{\partial^n \phi_2}{\partial \ugo^{n-1}} &
               \cdots &\frac{\partial^{n-1} \phi_{n-1}}{\partial \ugo^{n-1}}
\end{matrix}
\right|,
\end{gather*}
with $\psi_1 \equiv 1$, the vanishing of $\psi_{n}$ at a point $P$ of the
hyperelliptic curve $X_g$ is a necessary and sufficient
condition for $\omega(n\cdot P)$
to belong to $\WW_{g-1}$, where $\omega$ is
the Abel map $\omega:X_g\to \mathcal{J}_g$.
Here $\varepsilon_{n,g}$ is a plus or minus sign.

Further let $n(\ge g)$, $k(<g)$ and $\ell:=g-k-1$ be non-negative integers.
For a hyperelliptic curve $X_g$,
 the vanishing of
$\psi_{n+\ell}$, $\ldots$, $\psi_{n+1}$,
$\psi_{n}$, $\psi_{n-1}$,
$\ldots$,
$\psi_{n-\ell}$,
at a point $P$ of $X_g$ is a necessary and sufficient
condition for $\omega(n\cdot P)$
to belong to $\WW_{k}$.
\end{theorem}

Cantor gave a Brioschi-type expression for
the $\psi_n$-function \cite{C,Ma1}, 
$$
	\psi_n(u) =\left\{\begin{array} {llll}
   \varepsilon_{n,g}'
   (2y)^{n(n-1)/2}\cdot T_{(n-g-1)/2}^{(g+2)}( y, \frac{d}{d x}) &
         \text{ for $n\not\equiv g$ mod $2$}\\[1.5ex]
  \varepsilon_{n,g}'
   (2y)^{n(n-1)/2}\cdot T_{(n-g)/2}^{(g+1)}( y, \frac{d}{d x})
         &\text{ for $n\equiv g$ mod $2$.}
 \end{array} \right. 
$$
Here  $\varepsilon_{n,g}'$ is a plus or minus sign and 
$T_n^{(m)}$ is a  Toeplitz determinant \cite{Ma1}, 
$$
T_{n}^{(m)}\left( g(s), \frac{d}{ds}\right)
=\left| \begin{matrix}
g^{[m+n-1]} &g^{[m+n-2]} & \cdots & g^{[m+1]}& g^{[m]}\\
g^{[m+n]}   &g^{[m+n-1]} & \cdots & g^{[m+2]}& g^{[m+1]}\\
\vdots&\vdots&\ddots&\vdots&\vdots\\
g^{[m+2n-3]} &g^{[m+2n-4]} & \cdots & g^{[m+n-1]}& g ^{[m+n-2]}\\
g^{[m+2n-2]} &g^{[m+2n-3]} & \cdots & g^{[m+n]}& g ^{[m+n-1]}
\end{matrix} \right|, 
$$
and $T_{\ell}^{(m)}\left( g(s), \dfrac{d}{ds}\right)\equiv 1$
when $m$ is a non-negative integer and $\ell$ is a negative
integer, 
$g(s)$ is a function of an argument $s$ and
$$
	g^{[n]}(s):=\frac{1}{n!}\frac{d^n}{d s^n} g(s).
$$

We note for $y^2 =f(x)$ that $y^{2n-1} d^n y/d x^n$ is a polynomial in $x$
and coprime (in the sense of not vanishing together
on a point of the curve) to $f(x)$ in general.
Hence the function $y^{n(2m+2n-3)} T^{(m)}_n(y,\frac{d}{dx})$, that is, 
$$
\left| \begin{matrix}
          y^{2m+2n-3} y^{[m+n-1]} &y^{2m+2n-5} y^{[m+n-2]}
 & \cdots &y^{2m+1}  y^{[m+1]}& y^{2m-1} y^{[m]}\\
        y^{2m+2n-1} y^{[m+n]} & y^{2m+2n-3} y^{[m+n-1]}
& \cdots & y^{2m+3} y^{[m+2]}& y^{2m+1} y^{[m+1]}\\
\vdots&\vdots&\ddots&\vdots&\vdots\\
y^{2m+4n-7} y^{[m+2n-3]} &y^{2m+4n-9} y^{[m+2n-4]} & \cdots
          & y^{2m+2n-3}  y^{[m+n-1]}
& y^{2m+2n-5}  y ^{[m+n-2]}\\
y^{2m+4n-3} y^{[m+2n-2]} &y^{2m+4n-7} y^{[m+2n-3]}
 & \cdots & y^{2m+2n+1}  y^{[m+n+1]}     & y^{2m+2n-3} y ^{[m+n-1]}
\end{matrix} \right| 
$$
is an element of $\CC[x]$ and coprime to $y^2=f(x)$.
Hence $\psi_n(u)$ can be expressed by
$$
	\psi_n =\left\{
	 \begin{array}{lll}
         (2 y)^{g(g+1)/2} \alpha_n(x)
               & \text{ for } n-g=\text{odd, }  n > g + 1\\[1.0ex]
         (2 y)^{g(g-1)/2} \alpha_n(x) 
               & \text{ for } n-g=\text{even, } n > g  + 1\\[1.0ex]
         (2 y)^{n(n-1)/2} \alpha_n(x) & \text{ otherwise, } 
     \end{array} \right.
$$
where $\alpha_n(x)$ is a polynomial of $x$ and coprime of $y$
for $n > g + 1$, and $\alpha_n = \varepsilon_{n,g}' 1$ for $n \le g + 1$.
As shown in \cite{C}, the degree of $\alpha_n(x)$, $(n \ge g + 2)$ is 
$$
	\deg(\alpha_n) =\left\{ 
	\begin{array}{lll}
           \displaystyle{\frac{g(n+g)(n-g)-g(2g+1)}{2}}
               & \text{ for } n-g=\text{odd }\\[1.5ex]
           \displaystyle{\frac{g(n+g)(n-g)}{2}}
          & \text{ for } n-g=\text{even. }
     \end{array} \right.
$$

\begin{definition}
We define
$$
	\Phi_n:=\{P \in X_g \ |\ P\mbox{ is a zero of }\alpha_n \},
$$
and for $n > g$,
$$
	\Xi_n:=
	\Phi_{n-g+1}
	\cap\ldots
	\cap\Phi_{n-1}
	\cap\Phi_{n}
	\cap\Phi_{n+1}
	\cap\ldots
	\cap\Phi_{n+g-1}.
$$
\end{definition}

One should note here that there is no guarantee that $\Xi_n$ is not empty.

\begin{theorem}\label{thm:pToda}
Let $2N \ge 2g+1$.
For a hyperelliptic curve of genus $g$ which has
a point $(x_1', y_1') \in \Xi_{2N}$,
$\cV(u)$ in Theorem \ref{thm:Toda}
is the periodic solution of the Toda lattice equation
such that $\cV(u)=\cV(u + N c)$ with $c = 2 w(x_1', y_1')$.
\end{theorem}

\begin{remark}{\rm{
\begin{enumerate}
\item 
In the $g=1$ case, 
since an elliptic curve is a divisible group, there exist a 
 point $(x_1', y_1') $ which is a zero of a $\psi_{N}$;
as mentioned in Remark \ref{rmk:fay_g1}, for
every point $(x_1', y_1')$ we have a point $(x_2', y_2') = 2(x_1', y_1')$.
Thus, $\cV$ in Theorem \ref{thm:Toda},
is a periodic solution of the Toda lattice equation, provided
$\cV(u)=\cV(u + N c)$ with $c =  w(x_2', y_2')$.

\item 
We do not address in the present work 
the important problem of finding real-valued
solutions $\cV(u)$.
\end{enumerate}
}}
\end{remark}

\begin{example} 
\noindent{\rm{Case $g=1$, $N=3$ and $N=4$.

We consider the elliptic curve $y^2 = x^3 - x$ (which has an extra
automorphism of order two -- a property that is not related
to points of finite order but is usually accompanied
by a large number of exact solutions for the coordinates
of points on the curve that satisfy
algebraic conditions).
The division polynomials are given by
\begin{align*}
\psi_1 &= 1,
\\
\psi_2 &= -2 y,
\\
\psi_3 &=3 4y^2(x^4-6 x^2-1),
\\
\psi_4 &= -2y (x^2+1)(x^2+2 x-1)(x^2-2x-1),
\\
\psi_5 &=  
32x^{14}-187x^{12}-64x^{11}+2x^{10}
+320x^9-233x^8\\
&\quad+320x^7-52x^6-64x^5-61x^4+50x^2+1.
\end{align*}
 For $x_3'= (1/3)\sqrt{9+6\sqrt{3}}$, we have $N=3$ 
and for $x_3'=\sqrt{2}+1$ or a zero of $\psi_4$
we have $N=4$. 
}}
\end{example}


\begin{lemma}\label{lm:xN}
Let $2N \ge 2g+1$.
If $P:= (x_1', y_1')\in \Xi_{2N}$ has the property  
that $\ell P$ are distinct 
for different $\ell \in \{1, 2, \cdots, 2N\}$, then
$$
(x - x_1') (x - x_2')\cdots (x - x_{2N}') | \psi_{2N+m}(x, y)
$$
where $(x_\ell', y_\ell'):= \ell (x_1', y_1')$
and $m = -g+1, \ldots, 0, \ldots, 
g-1$.
\end{lemma}
\begin{proof}
By assumption, $\pm\ell P$ are exactly the points of the set $\Xi_{2N}$.
\end{proof}

\subsection{A periodic Toda lattice}

In this subsection,
 we consider the relation between Theorem \ref{thm:pToda}
and an algebraically completely integrable system\footnote{Different 
systems have been variously referred to in the literature
as ``periodic (generalized) Toda systems''; for extensive information
on definitions and solutions we refer to the monographs \cite{AvMV, V}.}
originally studied by  Kac and van Moerbeke \cite{KvM}.
Using  the solution given in Section \ref{Solution}, we find the 
explicit form of Flaschka's coordinates  for the Toda lattice.
The Hamiltonian of the Toda lattice equation is
$$
	H= \frac{1}{2}\sum_{k=1}^N P_k^2
           + \sum_{k=1}^N \exp( Q_k - Q_{k+1}),
$$
where $P_k = P_{k+N}$ and $Q_k = Q_{k+N}$.
For Flaschka's coordinates, $a_k = \exp( Q_k - Q_{k+1})$ and $b_k = -P_k$,
the equations of motion under $H$ become
\begin{equation}\left\{
\begin{array}{lll}
\displaystyle{\frac{d}{dt} a_k} = a_k (b_{k+1} - b_k),\\[2.0ex]
\displaystyle{\frac{d}{dt} b_k }= a_{k} - a_{k-1}.
\end{array}\qquad (k=1,2,\dots, N).\right.
\nonumber
\end{equation}

\begin{remark} 
The Toda Hamiltonian system admits the time inversion $t \mapsto -t$, and 
this corresponds to the hyperelliptic involution on $X_g$. 
\end{remark}

For brevity, we introduce the notation 
\begin{align*}
\sigma^{(n)}(t;t^\perp)&:=\sigma(t +nc + t^\perp), \quad
\sigma^{(c)}:=\sigma_{\flat}(c), \\
\zeta^{(n)}(t;t^\perp)&:=\sum_{i=1}^g {x_1'}^{i-1}\zeta_i(t +nc + t^\perp),
 \quad
\zeta^{(c)}:=\frac{1}{2}
 D_{1'} \log \sigma_{\flat}(c) ,\\
\wp^{(n)}(t;t^\perp)&:=\sum_{i,j=1}^g {x_1'}^{i+j-2}
\wp_{ij}(t +nc + t^\perp), \quad
\wp^{(c)}(t^{\perp}):= f_{1,2}(x_1'). 
\end{align*}

\begin{proposition} \label{prop:coef_anbn}
Using the sigma-function solution of the Toda lattice equation
in Theorem \ref{thm:Toda} (2), the
Flaschka coordinates for the Toda lattice are expressed as follows:
\begin{gather*}
\begin{split}
a_n &= \wp^{(n)}(t; t^\perp) - \wp^{(c)}(t^\perp) 
 = \frac{\sigma^{(n+1)}(t ; t^\perp) \sigma^{(n-1)}(t; t^\perp) }
{\sigma^{(n)}(t;t^{\perp})^2{\sigma^{(c)}}^2},\\
b_{n-1} &= 
D_t \log \frac{ \sigma^{(n)}(t ; t^\perp)}{
 \sigma^{(n-1)}(t; t^\perp) } - \zeta_c
= \zeta^{(n)}(t;t^\perp)- \zeta^{(n-1)}(t;t^\perp) -\zeta^{(c)}.
\end{split}
\end{gather*}
\end{proposition}

\begin{proof}
By definition of $P_k=-b_k$,
\begin{align*}
  P_{n-1}-P_n &= \zeta^{(n+1)}(t;t^\perp)
-2\zeta^{(n)}(t;t^\perp)
+\zeta^{(n-1)}(t;t^\perp),\\
   \cdots & \cdots \\ 
  P_2 - P_{3} &= \zeta^{(4)}(t;t^\perp)
-2\zeta^{(3)}(t;t^\perp)
+\zeta^{(2)}(t;t^\perp),\\
  P_1 - P_{2} &= \zeta^{(3)}(t;t^\perp)
-2\zeta^{(2)}(t;t^\perp)
+\zeta^{(1)}(t;t^\perp),\\
  P_0 - P_{1} &= \zeta^{(2)}(t;t^\perp)
-2\zeta^{(1)}(t;t^\perp)
+\zeta^{(0)}(t;t^\perp),\\
\end{align*}
with
\[
 - P_n  = \zeta^{(n+1)}(t;t^\perp)
-\zeta^{(n)}(t;t^\perp)
-(\zeta^{(1)}(t;t^\perp)
-\zeta^{(0)}(t;t^\perp)) - P_0.
\]
Since the total momentum should be invariant, we set
$$
  P_0  = -(\zeta^{(1)}(t;t^\perp) -\zeta^{(0)}(t;t^\perp)) +p_0.
$$
where $p_0$ is a constant corresponding to 
$\zeta^{(c)}$.
Then the equation $d b_k /dt=a_k-a_{k-1}$ holds
by Lemma \ref{lm:Dlogsigma} and Corollary \ref{cor:add},
equation $d a_k/dt=a_k(b_{k+1}-b_k)$ by Theorem \ref{thm:Toda}.
\end{proof}

\begin{proposition} \label{prop:coef_an}
The $a_n$'s and $b_n$'s given in Proposition 
\ref{prop:coef_anbn} are rational functions of
 $x_i$, $y_i$ $(i=1, \ldots, g)$ and 
$x'_1$, $y'_1$.
\end{proposition} 

\begin{proof} 
The translation law (\ref{eq:trans}) for the $\sigma_{\natural^n}$ 
functions shows that $a_n$ and $b_n$ are meromorphic
functions over the Jacobian $\JJ_g$;
Lemma \ref{lm:diff} gives 
$b_n$ as a meromorphic function over $S^g X_g \times X_g$,
derived using algebraic vector fields whose
coordinates are rational functions 
of $x_i$, $y_i$ $(i=1, \ldots, g)$, 
$x'_1$, $y'_1$. 

On the other hand, using Theorem \ref{thm:add}, 
$a_n$ is given by
\begin{gather*}
\begin{split}
&\frac{
\sigma(u +2n v^{[1]}+ 2v^{[1]}) 
\sigma(u +2n v^{[1]}- 2v^{[1]}) }
{ \sigma(u +2n v^{[1]})^2 \sigma_{\natural^n}(2v^{[1]})^2} 
\\
&=\lim_{(x_i',y_i') \to (x_1',y_1')} \left[
\frac{\prod_{i=0}^1 \Psi_{g+2n}
((x_1,y_1),\ldots, (x_{2n+2}', y_{2n+2}'), 
(x_1',(-1)^i y_1'),(x_2',(-1)^i y_2'))}
{(\Psi_{g+2n}((x_1,y_1)\ldots, (x_{2n+2}', y_{2n+2}'))
\Psi_{2}((x_1',y_1'), (x_2', y_2')))^2}\right.\\
&\qquad\qquad\qquad\times\left.
\prod_{i=1}^g\prod_{j=1}^2 
\frac{1}{ \Psi_2((x_i, y_i), (x_j', y'_j))}
\prod_{i=3}^{2n+2}\prod_{j=1}^2 
\frac{1}{ \Psi_2((x_i', y_i'), (x_j', y'_j))} \right].
\end{split}
\end{gather*}
Here 
$((x_1,y_1),\ldots, (x_{2n+2}', y_{2n+2}'))$ 
in both numerator and denominator of the equation
is an abbreviation for
$((x_1,y_1),\ldots, (x_g, y_g), (x_3',y_3'),\ldots, (x_{2n+2}', y_{2n+2}'))$.
This completes the proof of the statement.
\end{proof}

\begin{remark}
{\rm{
To exemplify Proposition \ref{prop:coef_an}
we write $b_n$ in the $g=1$ case:
$$
\zeta(u+v) - \zeta(u) - \zeta(v) =
\frac{y(u) - y(v)}{x(u)- x(v)}.
$$

}}
\end{remark}

Now we give an inverse of 
Theorem \ref{thm:pToda}; together, they constitute our
main result relating the periodic Toda system
to the sigma-function solution:
\begin{theorem} \label{thm:KvM_psi}
Let $2N \ge 2g+1$. If
 a hyperelliptic curve of genus $g$ has
a point $(x_1', y_1') \in X_g \setminus \infty$ such that
$\cV(u)$ in Theorem \ref{thm:Toda}
is the periodic solution of the Toda lattice equation,
i.e., $\cV(u)=\cV(u + N c)$ with $c = 2 w(x_1', y_1')$, then
$(x_1', y_1')$ belongs to $\Xi_{2N}$.
\end{theorem}

\begin{proof}
The periodicity condition $\cV(u)=\cV(u + N c)$ is
equivalent to the periodicity of $a_n$ and $b_n$ due to
Kac and van Moerbeke \cite{KvM}. The expression for the Flaschka
coordinates in Proposition \ref{prop:coef_anbn} allows for
a prefactor $\ee^{{}^t\beta u}$ of
$\sigma(u)$, $\beta$ a constant vector in $\CC^g$,
in terms of which the periodicity of $a_n$ and $b_n$, 
$(a_n = a_{n + N}, b_n = b_{n + N})$, is equivalent to
the equality:
\begin{equation}
 \ee^{2N{}^t \beta  v^{[1]}}\sigma(u + 2N v^{[1]})  = \sigma(u)  
\label{eq:sigmaNsigma}
\end{equation}
for every $\pm u \in \WW_k$ $(k = 0, \ldots, g)$,
where $v^{[1]}= w(x_1', y_1')\neq 0$ (by assumption).
Noting $2N - (g - 1) > g$,
this implies in particular, setting $u=\pm k v^{[1]}\neq 0$,
$$
\ee^{2N{}^t \beta v^{[1]}} \sigma((2N\pm k ) v^{[1]}) 
= \sigma(\pm k v^{[1]}). 
$$ 
The right-hand side vanishes for $(k = 0, \ldots, g-1)$
since $u=\pm k v^{[1]}$
belongs to the (translated) theta divisor, whereas
$\sigma_{\sharp}(v^{[1]})$ does not vanish. 

Moreover, we recall that a point $(x,y)$ in $X_g$ with $y=0$ 
is such that $2 w(x,y) \in W_{g-1}$, which still satisfies
the conclusion since  $N>1$.
Hence (\ref{eq:sigmaNsigma}) implies that $(x_1', y_1')\in \Xi_{2N}$.
\end{proof}

\begin{remark} \label{rmk:KvM_psi}
{\rm{Given 
the connection between
Theorem \ref{thm:pToda} and theory of Kac and van Moerbeke
provided by Theorem \ref{thm:KvM_psi},
we note in addition:
since it is known that for certain multi-indices $\gamma =(
\gamma_1,...,\gamma_g)$ and for 
$\pm u \in\kappa^{-1} \WW_k$ $(k = 1, \ldots, g-1)$,
letting $\partial^\gamma :=
\partial_{u_1}^{\gamma_1}\ldots  \partial_{u_g}^{\gamma_g},$
the derivative $\partial^\gamma\sigma (u)$ vanishes \cite{O,MP10},
by differentiating the equation, 
the
$\sigma$ function satisfies $(k = 1, \ldots, \ell)$ for a suitable
$\ell$,
$$
\partial^\gamma \sigma((2N-\ell) v^{[1]}) 
=\partial^\gamma \sigma((2N-\ell-1) v^{[1]}) 
=\ldots = \partial^\gamma \sigma((2N+\ell) v^{[1]}) =0.
$$
As we have assumed that $v^{[1]}\in\Xi_{2N}$ does not vanish,
 the vanishing of $\sigma$ and its derivatives on multiples of $v^{[1]}$ 
is a condition of flag-variety type (cf. \cite{AM, AHM, FH, CK, Ko}). 
Studying the topology of these
Toda orbits of smaller (than generic) dimension
was the original motivation for our work, which we plan to use 
for a sequel to this paper.}}
\end{remark}

\subsection{Hyperelliptic curve $\hX_{g,N-1}$ for the
periodic Toda lattice}\label{kvm}

The Lax matrix for the periodic Toda lattice is now given by
$$
\cL:=\begin{pmatrix}
b_1   & 1    & 0    &\cdots&  a_N \hw^{-1}\\
a_1   & b_2  & 1    &\cdots&    0    \\
\vdots&\ddots&\ddots&\ddots&\vdots\\
0   & \cdots & a_{N-2} & b_{N-1} & 1 &      \\
\hw   & \cdots & \cdots & a_{N-1} & b_{N} \\
\end{pmatrix}.
$$
The characteristic equation for $\cL$ defines the hyperelliptic 
spectral curve:
$$
\det(\cL - z) = -
\left(\hw +\frac{\prod_{i=1}^N a_i}{\hw} - \cP(z)\right) = 0,
$$
which gives the affine curve $(\hat{w},z)$ of genus $N-1$,
\begin{equation}\label{eq:Todacurve}
\hat{X}_{g,N-1}\,:\, ~\hat{w}^2-\cP(z)\hat{w}+\prod_{i=1}^Na_i=0.  
\end{equation} 
Here $\cP$ is given by
$$
	\cP(z) := \tDelta_{1, N}(z) - \tDelta_{2, N - 1}(z),
$$
where
$$
\tDelta_{n,m}:=\left|\begin{matrix}
b_m -z  & 1    & 0    &\cdots&  0\\
a_m   & b_{m+1}-z  & 1    &\cdots&    0    \\
\vdots&\ddots&\ddots&\ddots&\vdots\\
0   & \cdots & a_{n-2} & b_{n-1}-z & 1 &      \\
0   & \cdots & \cdots & a_{n-1} & b_{n}-z \\
\end{matrix}\right|.
$$
$$
	\cP(z) := (-1)^N z^N +
       \sum_{k=1}{N}(-1)^{N+k} \cI_k z^{N-k}.
$$
The coefficients of the powers of $z$ are given by
\begin{gather}
\begin{split}
       \cI_1 &= \sum_{i=1}^N b_i, \quad 
      \cI_2 = \sum_{i>j} b_i b_j - \sum_{i=1}^N a_i, \quad \ldots\ldots \\
             &\cdots\cdots \quad\qquad  \cdots\cdots\\
       \cI_N &= \tDelta_{1,N}(0) - a_N\tDelta_{2,N-1}(0), \quad 
       \cI_{N+1} = \prod_{i=1}^N a_i. \\
\end{split}
\end{gather}
We refer to $\hat{X}_{g,N-1}$ as the periodic Toda curve.
Since the characteristic polynomial $\det(\cL-z)$ is invariant 
under the Toda flow, these coefficients  give the Hamiltonians, i.e.
\begin{equation}
	\frac{\partial}{\partial t} \cI_j = 0 \qquad (j=1,2,\ldots,N+1).
\label{eq:Pt_I_j}
\end{equation}
The set  $\{\cI_j:j=1,\ldots,N+1\}$ gives an involutive,
complete family of 
integrals of motion, 
which guarantees the complete integrability of the Toda lattice.
In particular, $\cI_1$ and $\cI_{N+1}$ can be expressed as follows:

\begin{example}
{\rm From the formulae of $(a_n,b_n)$ in Proposition \ref{prop:coef_anbn}, 
the following expression for two Hamiltonians in terms
of Abelian functions follows directly:
$$
\cI_1=\sum_i b_i = N \zeta^{(c)}, \qquad
\cI_{N+1}=\prod_{i=1}^N a_i = (\sigma^{(c)})^{-2N}.
$$}
\end{example}

\begin{lemma}\label{rational}
$\cI_j$ $(j=1,\ldots,N+1)$ can be expressed as
rational functions of $(x_1', y_1')$ only.
\end{lemma}

\begin{proof}
Condition (\ref{eq:Pt_I_j}) shows that 
$\cI_j$ $(j=1,\ldots,N+1)$ does not depend upon $(x_j, y_j)$,
$j = 1, \ldots, g$.
\end{proof}

\begin{proposition} \label{prop:w2}
The hyperelliptic curve $\hX_{g,N-1}$ of (\ref{eq:Todacurve}) 
admits a  morphism to the curve:
$$
w^2 = \prod_{i=1}^{2N}(z - z_i) = 
  \cP(z)^2 -4\prod_{i=1}^N a_i,
$$
with $w := 2 \hw  -\cP(z).$
The coordinates of the
Weierstrass points of $\hX_{g,N-1}$ are rational functions of
$x_1'$ and $y_1'$.
\end{proposition}

\begin{proof}
The Weierstrass points have $z$-coordinates which correspond
to $w=0$. The coefficients of the corresponding polynomial
in $z$ are rational functions of $x_1'$ and $y_1'$ by
Lemma  \ref{rational}, and the fundamental theorem of
elementary symmetric functions gives the conclusion.
\end{proof}

\begin{example}
{\rm{For  $g=1$,
we consider again the curve $y^2 = x^3 - x$ or $\hy^2 = 4(x^3 - x)$.
Then,
\begin{gather}
\begin{split}
a_0(t) &= x - x_3',\\
a_1(t) &=\frac{ (3x{x_{3}'}^2-x-{x_{3}'}^3-2\hy{\hy_{3}'}-{x_{3}'})
}{(x-{x_{3}'})^2},\\
a_2(t) &= 2(x-{x_{3}'})\Bigr(
\frac{({\hy_{3}'}^4 +8 \hy{\hy_{3}'}^3
    + (-36x{x_{3}'}^2+(2(12+9x^2)){x_{3}'}-12x){\hy_{3}'}^2 }
{(-3{x_{3}'}^3+{x_{3}'}+2{\hy_{3}'}^2-2\hy{\hy_{3}'}+3x{x_{3}'}^2-x)^2}\\
 &+\frac{x^2+9{x_{3}'}^6-6{x_{3}'}^4+{x_{3}'}^2+12x{x_{3}'}^3
  -2x{x_{3}'}-18x{x_{3}'}^5-6x^2{x_{3}'}^2+9x^2{x_{3}'}^4 }
{(-3{x_{3}'}^3+{x_{3}'}+2{\hy_{3}'}^2-2\hy{\hy_{3}'}+3x{x_{3}'}^2-x)^2}
\Bigr),\\
\end{split}
\end{gather}
\begin{gather}
\begin{split}
b_0(t) = \frac{\hy - \hy'_3}{x - x_3'}, \quad
b_1(t) = \frac{\eta_{\hy,1} - \hy'_3}{\eta_{x,1} - x_3'}, \quad
b_2(t) = \frac{\eta_{\hy,2} - \hy'_3}{\eta_{x,2} - x_3'}, \quad
\end{split}
\end{gather}
where 
$(\eta_{x,1},\eta_{\hy,1})$  
and $(\eta_{x,2},\eta_{\hy,2})$ are the solutions of the equations
$\mu_1((z,w); (x,\hy), (x_3',\hy_3'))=0$ and
$\mu_2((z,w); (x,\hy), 2(x_3',\hy_3'))=0$
in the variables $(z, w)$, with:
$$
\mu_1((z,w); (x,\hy), (x_3',\hy_3'))
:= \frac{({x_{3}'}w-{\hy_{3}'}z+z\hy_{1}-x_{}w+
x_{}{\hy_{3}'}-\hy_{}{x_{3}'})}
         {({x_{3}'}-x_{})},
$$
\begin{gather*}
\begin{split}
&\mu_2((z,w); (x,\hy), 2(x_3',\hy_3'))
:=\frac{-z x_{}^2+x_{}z^2+2x_{}{x_{3}'}w{\hy_{3}'}
-2z\hy_{}{x_{3}'}{\hy_{3}'}-z{x_{3}'}^2-3x_{}z^2{x_{3}'}^2
}
{x_{}-3x_{}{x_{3}'}^2+{x_{3}'}^3+2\hy_{}{\hy_{3}'}+{x_{3}'}}\\
&\qquad+\frac{
3z x_{}^2{x_{3}'}^2+x_{}{x_{3}'}^2
+x_{}{x_{3}'}^4
-z {x_{3}'}^4-2w x_{}^2{\hy_{3}'}+2\hy_{} z^2{\hy_{3}'}
+z^2{x_{3}'}^3-x_{}^2{x_{3}'}^3+{x_{3}'}z^2-{x_{3}'}x_{}^2
}
{x_{}-3x_{}{x_{3}'}^2+{x_{3}'}^3+2\hy_{}{\hy_{3}'}+{x_{3}'}}.
\end{split}
\end{gather*}
}}
\end{example}

\bigskip
\bigskip
\subsection{Remark on a relation between $\hX_{g,N-1}$ and $X_g$}
Even though the Toda time flows have been linearized on
both the Jacobians of the curves $X_g$ and
$\hX_{g,N-1}$ (cf. Remark \ref{comparison}), and
solutions have been expressed in terms 
of hyperelliptic abelian functions on each,
the relationship between these curves is non-trivial.
To name only one other example, the Kowalevski
solution of the top has been compared 
by several authors with the Lax-pair solution,
and the classical Arithmetic-geometric Mean has been shown
to relate the (genus-2) curves \cite{LM}, the Jacobian of the one is the
quotient of the Jacobian of the other by a group of order 4.
A different Lax pair provides a curve of genus 3,
which covers an elliptic curve: the Prym variety of
the cover is again isogenous to the genus-2 Jacobian \cite{Markushevich}.

In our case, we make some observations.
The coefficients of  $\hX_{g,N-1}$ are given by hyperelliptic 
$\zeta$- and $\sigma$- functions,
which are transcendental functions of their arguments
(points of the Jacobian). However,
Propositions \ref{prop:coef_anbn} and \ref{prop:coef_an} 
show that they are rational functions 
of $x_i$'s and $y_i$'s; moreover, these coefficients are
Hamiltonians of motion, so the $t$-dependence of the
$a_n$ and $b_n$ fixes the curve $\hX_{g,N-1}$,
associated to a point on the Jacobian of $X_{g}$.

We  note that the cyclic group $C_N$ 
acts on the sets of
$a$'s and $b$'s by: $a_n \mapsto a_{n+1}$ and $b_n \mapsto b_{n+1}$,
via the addition on $2P$ (the image of the point $P$ 
in the Jacobian has order $2N$): $2\ell P \mapsto (2\ell + 2) P$ and 
$(2\ell + 1) P \mapsto (2\ell + 3) P$,
and that the curve $\hX_{g,N-1}$ is invariant
under this action,
although this does not guarantee that $\hX_{g,N-1}$
admits a $C_N$ action. There could be a relationship
of Kowalevski type between the two Jacobians
(which have different dimensions in general), such as 
quotienting  by the group $C_N$.  We plan to investigate
this relationship, starting with small-genus examples.

For example, in genus 1,
the relation between $X_1$ and $\hX_{1,N-1}$ could
give another solution to Poncelet's
closure problem (for Cayley's solution, cf. \cite{DR, GH}):
as we review in the Appendix, the periodic Toda flow
corresponding to a point of order $2N$ on a given elliptic
curve, which plays the role of our $X_g$, 
also corresponds to a closed Poncelet polygon with 
$2N$ sides; in this case, as we saw in subsection \ref{kvm} 
the periodic Toda curve
has genus $N-1$, and could be viewed as an algebraic solution to
the porism. 

We conclude the Appendix by a reference to another
classical problem solved in terms of transcendental functions
over an algebraic curve,
which therefore could play the role of  $X_g$, 
a construction similar to ours yielding
algebraic solutions over 
$\hX_{g,N-1}$.

\setcounter{section}{0}
\renewcommand{\thesection}{\Alph{section}}
\section{Appendix: Toda lattice and Poncelet's closure problem}

Poncelet's porism (cf. \cite{DR, GH}) can be stated as follows:

\begin{theorem}{\rm{(Poncelet)}}
Let $C$ and $D$ be two smooth conics in the real affine plane,
 such that $C$ includes $D$.
For an integer $N>2$,
if there exists a closed $N$-polygon inscribed  in $C$
and circumscribed about $D$, for every point $P$ in $C$ there
exists a polygon whose vertices are in $C$ and includes $P$,
and segments are tangent to $D$.
\end{theorem}

More generally, in the complex projective plane, 
the existence of such an $N$-sided Poncelet polygon
corresponds to a point of order $N$ on the elliptic
curve determined by the two conics, and a point in the
incidence correspondence of points of $C$ and tangents to $D$.
This interpretation in terms of a transcendental problem
is one of the deeper ways to prove the theorem (cf. \cite{Griffiths}),
whereas Poncelet used elementary projective geometry,
a subject that did not exist at the time.

As can be expected, a point of finite order gives rise to many
applications in the theory of periodic motion, and recently in
\cite{BZ} it was applied in a novel way to
give a condition of Fritz John type on a Dirichlet
problem for a planar domain bounded by an ellipse;
in the same paper, the authors give the following
explicit parametrization of the conics, and show
that the vertices of a Poncelet $N$-gon give a solution 
to the $N$-periodic Toda chain, which we reproduce below.

Without loss of generality,
we assume  the conic $C$ to be given by the equation $x^2 = y z$
and parametrize it by $(x, x^2, 1)$. Let the vertices of the Poncelet polygon
be the
$N$-points $(x_i^{(0)}, {x_i^{(0)}}^2, 1)$ $(i = 1, \ldots, N)$.
Let $D$ be defined by:
$$
	(x, y, z) A {}^t(x, y, z)= 0,
$$
where 
$$
	A = \begin{pmatrix}
              a_1 &  a_2 &  a_3 \\
              a_4 &  a_5 &  a_6 \\
              a_7 &  a_8 &  a_9 
	 \end{pmatrix}.
$$
Assume $a_5 =0$.
The dual conic $D^*$ of $D$ is given by 
$(X, Y, Z) A^{-1} {}^t(X, Y, Z) = 0$.
A pair $(P, L) \in C \times D^*$ such that $P \in L$ satisfies
$$
	x X + y Y + z Z = 0,
$$
for $P = (x, y, z)$ and $L = (X, Y, Z)$.
The relation is reduced to the elliptic curve $E_1$
$$
w^2 = \frac{1}{a_2 + a_4}(x, x^2, 1) A {}^t(x, x^2, 1),
$$
where
$$
  w = \frac{1}{\sqrt{\det A}} \left(h_1(x)\frac{Y}{X} - h_2(x) \right),
$$
and the $h_1,h_2$ are  polynomials in $x$.
Poncelet's closure
is equivalent to finding  
a matrix $A$ as above and a point $(x, w)$ belonging to $E_1$ 
such that it satisfies the equation of  Kiepert and Brioschi,
$$
\psi_N((x,w)) = 0,
$$
a criterion attributed to Cayley and proved in \cite{GH}. 
We regard the equation of $\psi_N((x,w)) = 0$ as the moduli
equation for a given $x$.

For such an $A$, Poncelet's theorem means 
that the vertex $P_n \equiv (x_n, x_n^2, 1) \in C$ $(n = 1, 2, \ldots, N)$
satisfies the periodic Toda lattice,
$x_n = \wp((n - 1) u_0+t)$
\cite[\S7.1]{BZ},
\begin{gather}
\begin{split}
        -\frac{d^2}{d t^2}& \log [\wp(n u_0+t)-\wp(u_0)]
        = [\wp((n+1)u_0+t)-\wp(u_0)] \\
&- 2[\wp(n u_0+t)-\wp(u_0)]
+ [\wp((n-1)u_0+t)-\wp(u_0 )],
\end{split}
\end{gather}
where
\begin{equation}
     u_0 = \int^{(x^{(0)}_1,w^{(0)}_1)}_\infty 
 \frac{d x}{2 w}
\end{equation}
(in the previous notation,  $x_n^{(0)} = \wp((n-1)u_0)$).

Lastly, we cite the problem of finding
the general  solution of the fifth-degree algebraic equation
with Galois group $\mathrm{PSL}(2, \FF_5) = A_5$, alternating group on five elements.
While the solution cannot be algebraic in the coefficients
of the polynomial equation,
classical authors such as Jacobi, Galois, and Klein \cite{Kl},
gave a solution in terms of the zeros of $\psi_5$ for an elliptic curve
$X_1$\footnote{More recently, H. Umemura \cite[Chapter III.c]{Mu}
gave a solution of an algebraic equation of any degree $n$ in terms of theta
constants of a hyperelliptic curve of genus $n$, in other words,
in terms of a Siegel modular function.}. 
In the classical, fifth-degree case,
Humbert \cite{H} expressed the period-5 condition in
terms of a Poncelet pentagon, but in addition proved that it is
equivalent to the curve $y^2=(x-x_1)\cdots (x-x_5)$
having real multiplication by the quadratic order of 
discriminant 5.
Hashimoto and Sakai \cite{HS} expressed the general condition for 
a hyperelliptic curve of genus 2, $H_{1,2}$,
$y^2=(x-x_1)\cdots (x-x_5)\cdot (x-x_6)$
to have  real multiplication of discriminant 5 
again translating it into 
 a condition for the existence of 
closed Poncelet pentagons. Mestre \cite{Mestre} 
 generalized the condition of real multiplication to genus $g$.
  Lemma \ref{lm:xN} gives a criterion for the
points of order $2N$ in $X_g$ in terms of the (genus-$g$)
division polynomial;
 by analogy $H_{1,2}$ might be replaced by 
a hyperelliptic curve $H_{g,N-1}$, 
$$
\ty^2 = (x - x_1') (x - x_2') \cdots (x - x_{2N}').
$$
Proposition \ref{prop:w2} gives an algebraic relation
between
$H_{g,N-1}$ and the moduli of $\hX_{g,N-1}$.

\bigskip

Yuji Kodama
Department of Mathematics
The Ohio State University
Columbus, OH 43210,
U.S.A.

{kodama@math.ohio-state.edu}

\medskip
Shigeki Matsutani:
8-21-1 Higashi-Linkan, Minami-ku
Sagamihara 252-0311,
JAPAN

{rxb01142@nifty.com}

\medskip
Emma Previato:
Department of Mathematics and Statistics,
Boston University,
Boston, MA 02215-2411,
U.S.A.

{ep@bu.edu}


\begin{thebibliography}{BBEIM} 
\bibitem[AM]{AM}
  \by{M. Adler, and P. van Moerbeke}
  \paper{The Toda Lattice, Dynkin diagrams, singularities and
Abelian varieties}
\jour{Invent. Math.}
\vol{103} \yr{1991} \pages{223-278}.

\bibitem[AHM]{AHM}
  \by{M. Adler, L. Haine, P. van Moerbeke}
  \paper{Limit matrices for the Toda flow and periodic flags
 for loop groups}
\jour{Math. Ann.}
\vol{296} \yr{1993} \pages{1-33}.

\bibitem[AvMV]{AvMV}
\by{M. Adler, P. van Moerbeke and P. Vanhaecke} \book{Algebraic integrability,
Painlev\'e geometry and Lie algebras} Ergebnisse der Mathematik und
ihrer
Grenzgebiete. 3. Folge. A Series of Modern Surveys in Mathematics,
   \textbf{47}. Springer-Verlag, Berlin, 2004.



\bibitem[B1]{B1}
  \by{H.F. Baker}
  \book{Abelian Functions, Abel's theorem and the allied theory of theta
functions}
Cambridge University Press, Cambridge, 1897.

\bibitem[B2]{B2}
  \by{H. F. Baker}
  \paper{On the hyperelliptic sigma functions}
\jour{ Amer. J.  Math}
\vol{ 20} \yr{1898} \pages{ 301-384}.

\bibitem[B3]{B3}
  \by{H. F. Baker}
  \paper{On a system of differential equations
leading to periodic functions}
  \jour{Acta Math.}
  \vol{27}
  \yr{1903}
  \pages{135-156}.

\bibitem[BES]{BES}
  \by{E.D. Belokolos, V.Z. Enolskii  and M. Salerno}
\paper{Wannier functions for quasi-periodic finite-gap
    potentials}
\jour{Theor.  Math. Phys.} \vol{144} \yr{2005}
\pages{1081-1099}.

\bibitem[Br]{Br}
  \by{F. Brioschi}
   \paper{Sur quelques formules pour la multiplication des fonctions
          elliptiques}
   \jour{C. R. Acad. Sci. Paris}
\vol{59}
\yr{1864} \pages{769-775}.



\bibitem[BEL]{BEL}
  \by{ V. M. Buchstaber, V. Z.  Enolskii, and  D. V. Leykin}
  \paper{Kleinian Functions, Hyperelliptic Jacobians and Applications}
  \book{Reviews in Mathematics and Mathematical Physics (London)}
  \eds{\ S.P. Novikov and I.M. Krichever,}
  \publ{Gordon and Breach} \publaddr{India} 1997 \pages{1-125}.

\bibitem[BZ]{BZ}
  \by{V. P. Burskii and A. S. Zhedanov}
  \paper{On Dirichlet, Poncelet and Abel problems}
\texttt{arXiv:0903.253}.

\bibitem[C]{C}
  \by{D.G. Cantor}
  \paper{On the analogue of the division polynomials
         for hyperelliptic curves}
  \jour{J. reine angew. Math.}
  \vol{447}
  \yr{1994}
  \pages{91-145}.

\bibitem[CK]{CK}
   \by{L. Casian, Y. Kodama}
   \paper{Compactification of the isospectral varieties of nilpotent
    Toda lattices}
  \jour{RIMS Proceedings (Kyoto University). Surikaisekiken Kokyuroku}
   \vol{1400} \yr{2004} \pages{39--87}
     mathAG/0404345.

\bibitem[DR]{DR}
\by{V. Dragovi\'c and M. Radnovi\'c}
\paper{Cayley-type
  conditions for billiards within $k$ quadrics in $\mathbb{R}^d$}
\jour{J. Phys. A}
  \vol{37}
  \yr{2004} no. 4 \pages{1269--1276}

\bibitem[EEG]{EEB}
\by{J.C. Eilbeck, V.Z. Enolski and J. Gibbons},
   \paper{Sigma, tau and Abelian functions of algebraic curves}
J. Phys. A  \vol{43}
  \yr{2010} no. 45 \pages{455216}, 20 pp.


\bibitem[EEL]{EEL}
   \by{J.C. Eilbeck, V.Z. Enolskii and D.V. Leykin},
   \paper{On the Kleinian construction of Abelian functions of canonical
   algebraic curves}
Proceedings of the 1998 SIDE III conference (2000).

\bibitem[EEM\^OP1]{EEMOP1}
 \by{J.C. Eilbeck, V.Z. Enol'skii, S. Matsutani, Y. \^Onishi, and E. Previato}
\paper{Addition formulae over the Jacobian pre-image of hyperelliptic 
Wirtinger varieties} 
\jour{J. Reine Angew. Math.}
\vol{619} \yr{2008}  \pages{37-48}.

\bibitem[EEM\^OP2]{EEMOP2}
 \by{J.C. Eilbeck, V.Z. Enol'skii, S. Matsutani, Y. \^Onishi, and E. Previato}
\paper{Abelian functions for trigonal curves of genus three}
\jour{Int. Math. Res. Notices}
 \yr{2008}, no. 1, Art. ID rnm 140, 38 pp.


\bibitem[EG]{EG}
\by{V.Z. Enolski and J. Gibbons}
Addition theorems on the strata of the theta divisor of genus three 
hyperelliptic curves, draft (2002).

\bibitem[F]{F}
  \by{J.D. Fay}  \book{Theta functions on Riemann Surfaces}
Lecture Notes in Mathematics, Vol. 352. Springer-Verlag, Berlin-New York,
  1973.


\bibitem[FH]{FH}
\by{H. Flaschka and L.  Haine}
\paper{Vari\'et\'es de drapeaux et r\'eseaux de
Toda}  \jour{C. R. Acad. Sci. Paris}
S\'er. I Math.  \vol{312}  (1991)  no. 3 \page{255-258}.


\bibitem[GHMT]{GHMT}
\by{F. Gesztesy, H. Holden, J. Michor and G. Teschl} 
 \book{Soliton equations and their algebro-geometric
 solutions. Vol. II. $(1+1)$-dimensional discrete models}
 Cambridge Studies in
 Advanced Mathematics, \textbf{114}.
 Cambridge University Press, Cambridge, 2008.

\bibitem[G]{Griffiths}
\by{P.A. Griffiths}
\paper{Variations on a theorem of Abel}
\jour{Invent. Math.}
\vol{35}
\yr{1976} \pages{321--390}.

\bibitem[GH]{GH}
  \by{P. Griffiths and J. Harris}
  \paper{On Cayley's explicit solution to Poncelet's porism}
\jour{Enseign. Math.}
\vol{24} \yr{1978} \pages{31-40}.

\bibitem[Hi]{Hi}
  \by{R. Hirota} 
  \paper{Soliton no suuri: Mathematics in soliton}
   (Japanese) Iwanami Tokyo 1992.
\book{The direct method in soliton theory} 
translated from the 1992 Japanese original and edited by Atsushi Nagai, Jon
  Nimmo and Claire Gilson. With a foreword by Jarmo Hietarinta and
  Nimmo. Cambridge Tracts in Mathematics, 
\textbf{155}. Cambridge University Press,
  Cambridge, 2004.


\bibitem[Hu]{H}
  \by{G. Humbert}
  \book{Sur les fonctions abeliennes singulieres}
\book{Oeuvres de G. Humbert 2, pub. par les soins de Pierre Humbert
et de Gaston Julia} Paris, Gauthier-Villars
\yr{1936} \pages{297-401}.

\bibitem[HS]{HS}
  \by{K. Hashimoto and Y. Sakai}
  \paper{General form of Humbert's modular equation for curves with
real multiplication of $\Delta = 5$}
\jour{Proc. Japan Acad.} \vol{85} \yr{2009} \pages{171-176}.


\bibitem[KvM]{KvM}
\by{M. Kac and P. van Moerbeke} On some periodic Toda
 lattices.  Proc. Nat. Acad. Sci. U.S.A.  72  (1975), 
no. 4, 1627--1629;  A complete solution of
 the periodic Toda problem,
\textit{ibid.}, 72  (1975),   no. 8, 2879--2880.


\bibitem[Ki]{Ki}
  \by{L. Kiepert}
  \paper{Wirkliche Ausf\"uhrung der ganzzahligen Multiplication der
 elliptischen
         Functionen} \jour{J. reine angew. Math.}
\vol{76}
         (1873) \page{21-33}.


\bibitem[Kl]{Kl}
\by{F. Klein}
\book{Vorlesungen \"uber das Ikosaeder und die Aufl\"osung der 
      Gleichungenvom f\"unften Grade}
\publ{Teubner} \publaddr{Leipzig} (1884).


\bibitem[Ko]{Ko}
  \by{Y. Kodama} 
  \paper{Topology of the real part of hyperelliptic Jacobian
associated with the periodic Toda lattice}
\jour{Teoret. Mat. Fiz.} \vol{133}  \yr{2002}  no. 3
  \page{439--462}.  Translation in \jour{Theoret. and Math. Phys.}
\vol{133}  \yr{2002}  no. 3
  \page{1692--1711}.

\bibitem[LM]{LM}
\by{F. Lepr\'evost and D. Markushevich}
\paper{A tower
of genus two curves related to the Kowalewski top}
\jour{J. Reine Angew. Math.} \vol{514}
\yr{1999} \page{103--111}.


\bibitem[McKvM]{McKvM}
\by{H.P. McKean and P. van Moerbeke}
\paper{ Hill and Toda curves}
 \jour{Comm. Pure Appl. Math.}  \vol{33}  
\yr{1980} no. 1 \page{23-42}. 

\bibitem[Mar]{Markushevich}
  \by{D. Markushevich} 
  \paper{Kowalevski top and genus-2 curves,
Kowalevski Workshop on Mathematical Methods of Regular Dynamics (Leeds, 2000)}
  \jour{ J. Phys. A} \vol{34} 
\yr{2001} \page{2125--2135}.

\bibitem[Mat1]{Ma1}
  \by{S. Matsutani} 
  \paper{Appendix in \cite{O}}
  \jour{Proc. Edinburgh Math. Soc.} \vol{48} \yr{2005} \page{736-742}.

\bibitem[Mat2]{Ma2}
  \by{S. Matsutani} 
  \paper{Hyperelliptic solution of KdV and KP equations:
re-evaluation of Baker's study on hyperelliptic sigma functions}
  \jour{J. Phys. A: Math. Gen.} \vol{34} \yr{2001} \page{4721-4732}.

\bibitem[Mat3]{Ma3}
  \by{S. Matsutani} 
  \paper{Toda Equation and $\sigma$-Function of Genera One and Two}
  \jour{J. Nonlinear Math. Phys.} \vol{10} \yr{2003} \page{555-561}.


\bibitem[Mat4]{Ma4}
\by{S. Matsutani} 
  \paper{Neumann system and hyperelliptic al
functions} \jour{Surv. Math. Appl.} \vol{3} \yr{2008} \page{13-25}.


\bibitem[MP1]{MP}
   \by{S. Matsutani and E. Previato}
   \paper{A generalized Kiepert Formula for $C_{ab}$ curves}
 \jour{Israel J. Math.} \vol{171}  \yr{2009} 
\page{305--323}.

\bibitem[MP2]{MP10} S. Matsutani and E. Previato,  
\paper{Jacobi inversion on strata of the Jacobian of the 
$C_{rs}$ curve $y^r = f(x)$. II}, arXiv:1006.1090.


\bibitem[Me]{Mestre} J.-F. Mestre,
\paper{Courbes hyperelliptiques \`a multiplications r\'eelles}
\jour{C. R. Acad. Sci. Paris S\'er. I Math.}
\vol{307}
\yr{1988}
\page{721--724}.

\bibitem[vM]{vM}
 \by{P. van Moerbeke} \paper{The spectrum of Jacobi matrices}
 \jour{Invent. Math.}  \vol{37}  \yr{1976} no. 1 \page{45-81}.

\bibitem[vMM]{vanMM}
 \by{P. van Moerbeke and D. Mumford}
\paper{The spectrum of
 difference operators and algebraic curves}
\jour{  Acta Math.}  \vol{143}  \yr{1979} no. 1-2
\page{ 93-154}.

\bibitem[Mu]{Mu}
 \by{D. Mumford}\book{Tata Lectures on Theta, vols. 1-2}
     \publ{Birkh\"auser} \publaddr{Boston} 1984.


\bibitem[N]{N}
    \by{A. Nakayashiki}
     \paper{Sigma function as a tau function}
\jour{Int. Math. Res. Notices}  \vol{2010}  \yr{2009} \page{373-394}.

\bibitem[\^O]{O}
  \by{Y. \^Onishi}
  \paper{Determinant expressions for hyperelliptic functions}
  \jour{Proc. Edinburgh Math. Soc.} \vol{48} \yr{2005} \page{705-742}.


\bibitem[S1]{Sch1}
R.J. Schilling,  \paper{Generalizations of the Neumann system. A curve-theoretical
approach. I} \jour{Comm. Pure Appl. Math.} \vol{40} 
\yr{1987} no. 4, 455-522. 

\bibitem[S2]{Sch2}
\by{R.J. Schilling}
\paper{Generalizations of the Neumann system. A curve theoretical
approach. II} \jour{Comm. Pure Appl. Math.} \vol{42} \yr{1989} no. 4, 409-442. 

\bibitem[S3]{Sch3}
\by{R.J. Schilling} \paper{Generalizations of the Neumann system---a curve
theoretical approach. III. Order $n$ systems}
 \jour{Comm. Pure Appl. Math.} \vol{45} \yr{1992} no. 7, 775-820.


\bibitem[TD]{TD}
  \by{S. Tanaka and E. Date}
  \book{KdV houteishiki}
Kinokuniya Press, Tokyo, 1979 (in Japanese).

\bibitem[T]{To}
M. Toda,
 Vibration of a Chain with Nonlinear Interaction,
{\it J. Phys. Soc. Japan} {\bf 22} (1967), 431--436;
Wave propagation in anharmonic lattices, \textit{ibid.} 
\textbf{23} (1967), 501-506.

\bibitem[V]{V}
P. Vanhaecke, 
 \book{Integrable systems in the realm of algebraic geometry} Second
edition. Lecture Notes in Mathematics, \textbf{1638}. 
Springer-Verlag, Berlin, 2001. 

\bibitem[W]{W1}
\by{K. Weierstrass}
\paper{Zur Theorie der Abelschen Funktionen}
\jour{J. Reine Angew. Math.} \vol{47} \yr{1854} 289-306.



\bibitem[WW]{WW}
\by{E.T. Whittaker and G.N. Watson}
\book{A Course of Modern Analysis, {\it{fourth edition}}}
\publ{Cambridge University Press} 1927.

\end{thebibliography}
\end{document}